\documentclass[12pt]{amsart}

\usepackage{url}
\usepackage[dvipsnames]{xcolor}
\usepackage{tikz}
\usepackage{enumitem}
\usetikzlibrary{calc}
\usepackage{afterpage}
\usepackage{float}
\usepackage{graphicx}
\usepackage{amsmath}
\usepackage{color}
\usepackage{caption}
\usepackage{subcaption}
\usepackage{amsfonts,amssymb,amscd}
\usepackage{mathrsfs}
\usepackage{bm}
\usepackage{verbatim} 
\usepackage{manfnt}
\usepackage{yhmath}

\usepackage[toc,page]{appendix}

\calclayout

\newtheoremstyle{remboldstyle}
  {}{}{\itshape}{}{\bfseries}{.}{.5em}{{\thmname{#1 }}{\thmnumber{#2}}{\thmnote{ (#3)}}}
\theoremstyle{remboldstyle}

\usepackage[outdir=./]{epstopdf}

\newtheorem{thm}{Theorem}[section]
\newtheorem{prop}[thm]{Proposition}

\newtheorem{cor}[thm]{Corollary}

\newtheorem{thmx}{Theorem}

\theoremstyle{definition}
\newtheorem{definition}[thm]{Definition}
\newtheorem{construction}[thm]{Construction}
\newtheorem{example}[thm]{Example}

\newtheorem{rem}[thm]{Remark}
\newtheorem{notation}[thm]{Notation}

\makeatletter
\DeclareFontFamily{U}{tipa}{}
\DeclareFontShape{U}{tipa}{m}{n}{<->tipa10}{}
\newcommand{\arc@char}{{\usefont{U}{tipa}{m}{n}\symbol{62}}}%

\newcommand{\arc}[1]{\mathpalette\arc@arc{#1}}

\newcommand{\arc@arc}[2]{%
  \sbox0{$\m@th#1#2$}%
  \vbox{
    \hbox{\resizebox{\wd0}{\height}{\arc@char}}
    \nointerlineskip
    \box0
  }%
}
\makeatother

\numberwithin{equation}{section}

\catcode`\@=11
\newdimen\cdsep
\def\cdstrut{\vrule height .6\cdsep width 0pt depth .4\cdsep}
\def\@cdstrut{{\advance\cdsep by 2em\cdstrut}}

\def\arrow#1#2{
  \ifx d#1
    \llap{$\scriptstyle#2$}\left\downarrow\cdstrut\right.\@cdstrut\fi
  \ifx u#1
    \llap{$\scriptstyle#2$}\left\uparrow\cdstrut\right.\@cdstrut\fi
  \ifx r#1
    \mathop{\hbox to \cdsep{\rightarrowfill}}\limits^{#2}\fi
  \ifx l#1
    \mathop{\hbox to \cdsep{\leftarrowfill}}\limits^{#2}\fi
}
\catcode`\@=12

\newcommand{\Chat}{\widehat{\mathbb{C}}}

\DeclareMathOperator{\interior}{int}
\DeclareMathOperator{\closure}{closure}

\DeclareMathOperator{\ratd}{\mathrm{Rat}_d}
\DeclareMathOperator{\pold}{\mathrm{Pol}_d}

\DeclareMathOperator{\poldmh}{\mathrm{Pol}_{d-1}^{HC}}
\DeclareMathOperator{\rat}{\mathrm{Rat}}
\DeclareMathOperator{\CP}{\mathrm{CP}}
\DeclareMathOperator{\CV}{\mathrm{CV}}

\makeatletter
\@namedef{subjclassname@2020}{\textup{2020} Mathematics Subject Classification}
\makeatother

\begin{document}


\title[Dynamics of Hyperbolic Schwarz Reflections]{Dynamics of Hyperbolic Schwarz Reflections}

\author{Kirill Lazebnik}
\thanks{The author was partially supported by NSF grant DMS-2541258.}






\begin{abstract} We give a conformal mating description for the class $\Sigma_d$ of hyperbolic Schwarz reflections with connected and full filled Julia set; this parallels the recent development in several non-hyperbolic settings. We show that the escaping dynamics of $\Sigma_d$ can be characterized independently by the class of \emph{annular Schwarz reflections} we introduce. Our main result is that each $\sigma\in\Sigma_d$ decomposes into a unique hyperbolic polynomial factor and a unique annular Schwarz reflection factor and, conversely, any two such factors may be mated to determine a unique such Schwarz reflection. We also show that annular Schwarz reflections converge to Nielsen maps in suitable settings, connecting the hyperbolic and non-hyperbolic mating descriptions.
\end{abstract}


\maketitle


\section{Introduction}

\subsection{Overview.}\label{overview} We denote by $\Chat$ the Riemann sphere.

\begin{definition} A domain $D\subset\Chat$ is called a \emph{Quadrature Domain} if the function $\iota:\partial D\rightarrow\Chat$ defined by $\iota(z):=\overline{z}$ continuously extends to a holomorphic function $\iota: D\rightarrow\Chat$, in which case $\iota$ is called the \emph{Schwarz Function}, and the conjugate $\sigma:=\overline{\iota}$ is called the \emph{Schwarz Reflection}.
\end{definition}

The simplest examples of Schwarz reflections are reflections in a Euclidean circle $C$, so that any Euclidean disc $D$ is a quadrature domain. The image $r(D)$ under any rational mapping $r$ which is injective in $D$ is also a quadrature domain. Quadrature domains and Schwarz reflections feature prominently in several areas of mathematics: function theory \cite{MR4381220}, operator theory \cite{MR1302653}, random normal matrix theory \cite{MR1986427}, potential theory \cite{MR1094715} as well as fluid dynamics \cite{Ric72}. Recently, it has been realized that a very profitable point of view is to study the dynamics of Schwarz Reflections: both because these dynamical systems exhibit rich phenomena \cite{sabyaicmpaper}, but also because the dynamical perspective has produced important applications in the aforementioned areas \cite{MR3454377}, \cite{rashmita2025topologysingularitiesquadraturedomains}. 
  
Let us consider an example (following Section 4.1.1 of \cite{MR5098194}), and illustrated in Figures \ref{pinchedfig} and \ref{nonpinchedfig}.
  
\afterpage{  
\begin{figure}[ht]
    \centering
    
\begin{minipage}{\textwidth}
\centering    

    \begin{subfigure}[b]{0.32\textwidth}
        \centering
        \includegraphics[width=\linewidth]{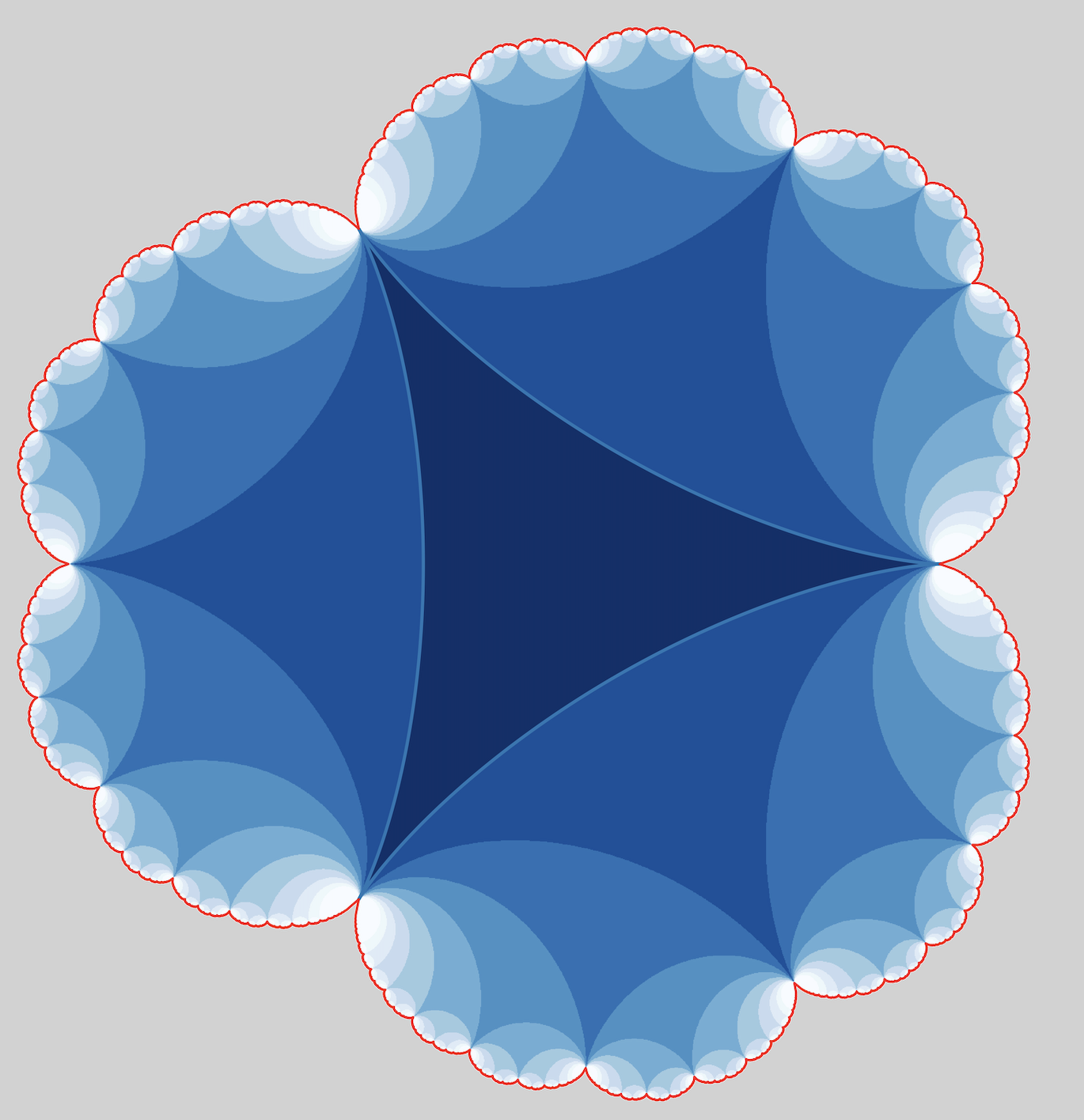}
        \caption{$\sigma_2: r_2(\mathbb{D}^*) \rightarrow \Chat$}
        \label{fig:first1}
    \end{subfigure}
    \hfill
    \begin{subfigure}[b]{0.32\textwidth}
        \centering
        \includegraphics[width=\linewidth]{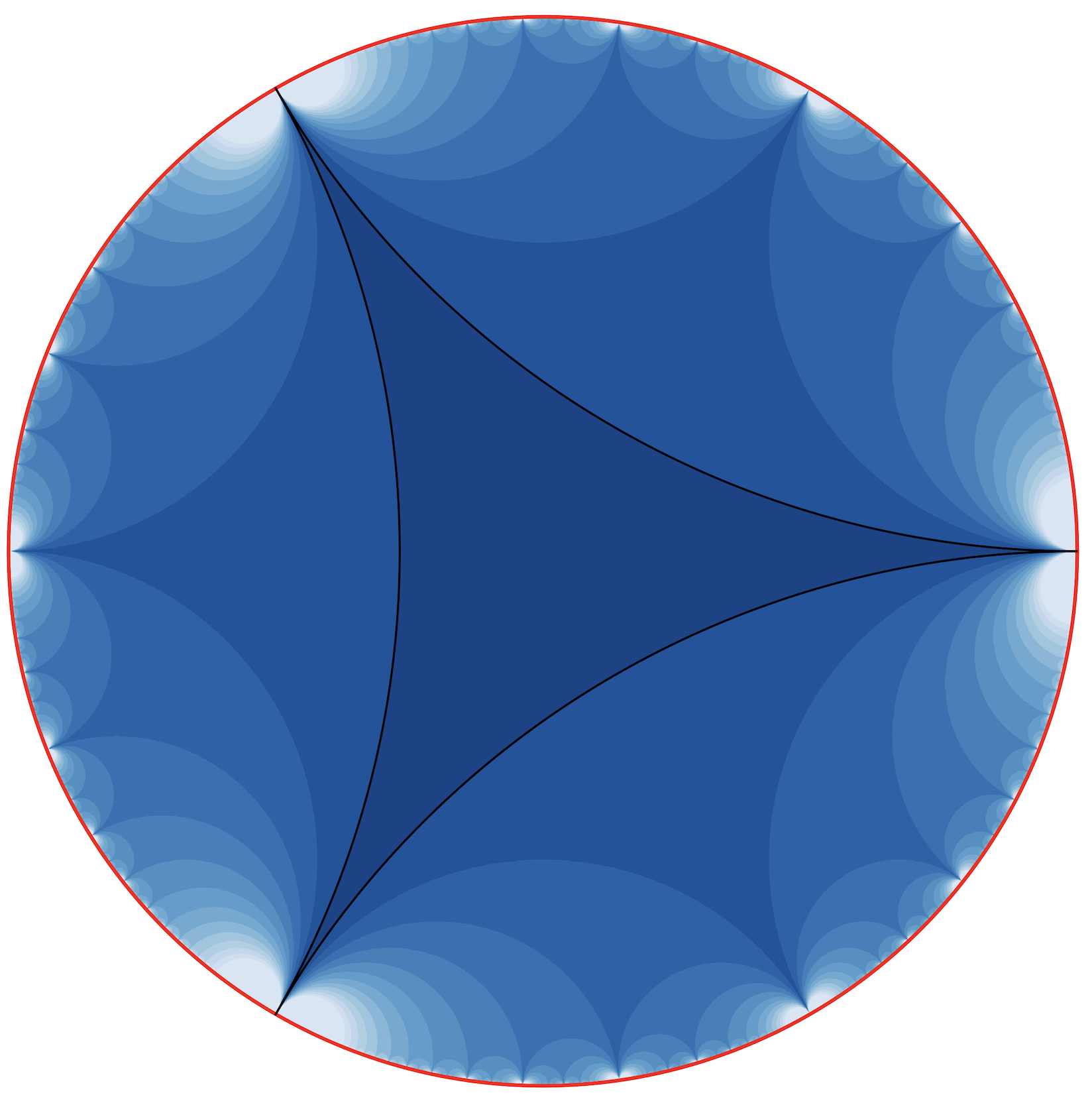}
        \caption{$\rho: \mathbb{D}\setminus T \rightarrow \mathbb{D}$}
        \label{fig:second1}
    \end{subfigure}
    \hfill
    \begin{subfigure}[b]{0.32\textwidth}
        \centering
        \includegraphics[width=\linewidth]{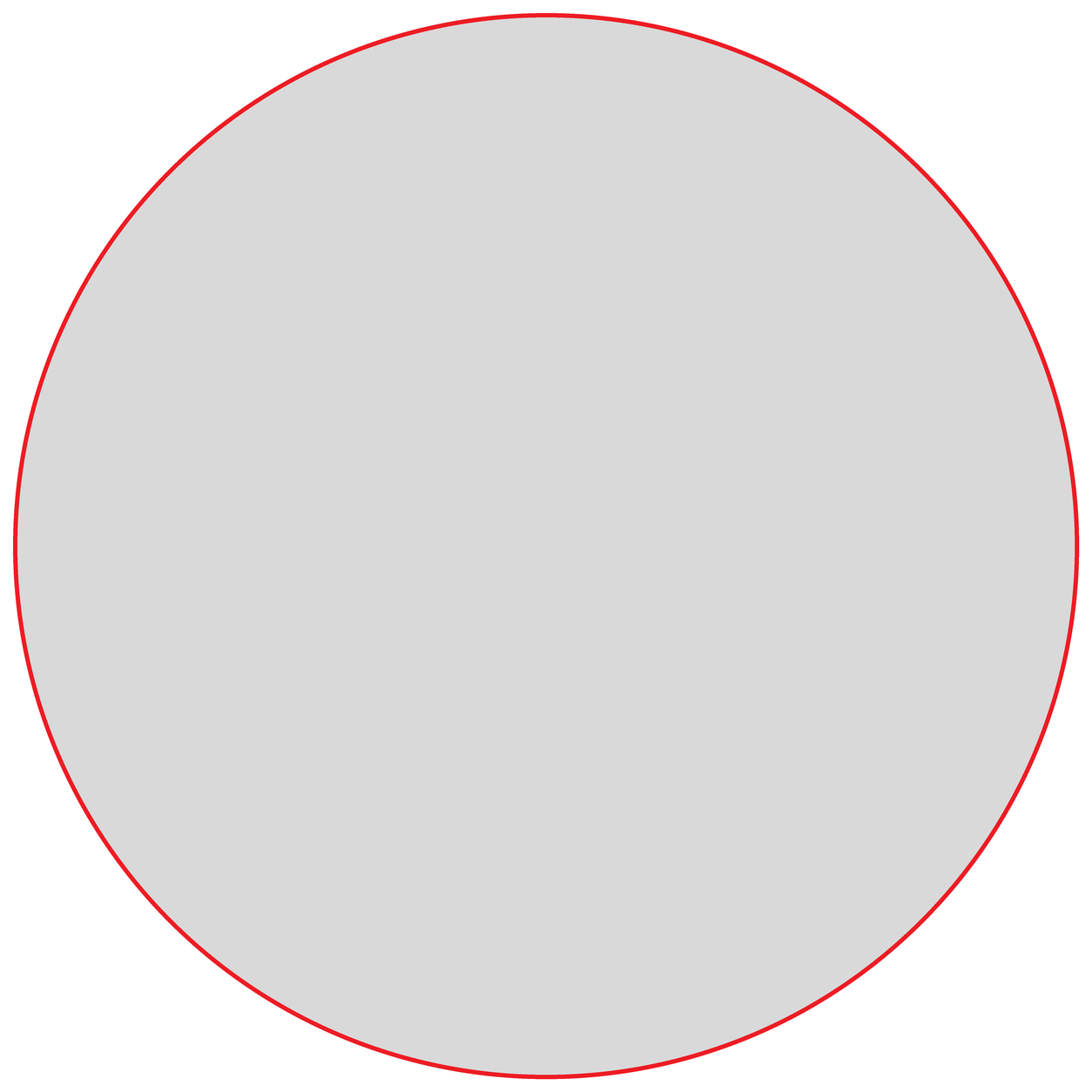}
        \caption{$\overline{z}^2: \mathbb{D} \rightarrow \mathbb{D}$}
        \label{fig:third1}
    \end{subfigure}
    \captionsetup{width=.95\textwidth}
    \caption{This Figure illustrates (A) the non-hyperbolic Schwarz reflection $\sigma_2$, as a mating of (B) the Nielsen map $\rho$, with (C) the polynomial $\overline{z}^2$; see Theorem 1.1 of \cite{MR4706575}. The darkest blue triangle in (A) is the deltoid $\Chat\setminus r_2(\mathbb{D}^*)$, and subsequent $\sigma_2$-preimages of the deltoid are shown in shades of lighter blue. The Julia set is in red.}
    \label{pinchedfig}
    \end{minipage}
\vspace{1cm}

\begin{minipage}{\textwidth}
    \centering
    \begin{subfigure}[b]{0.32\textwidth}
        \centering
        \includegraphics[width=\linewidth]{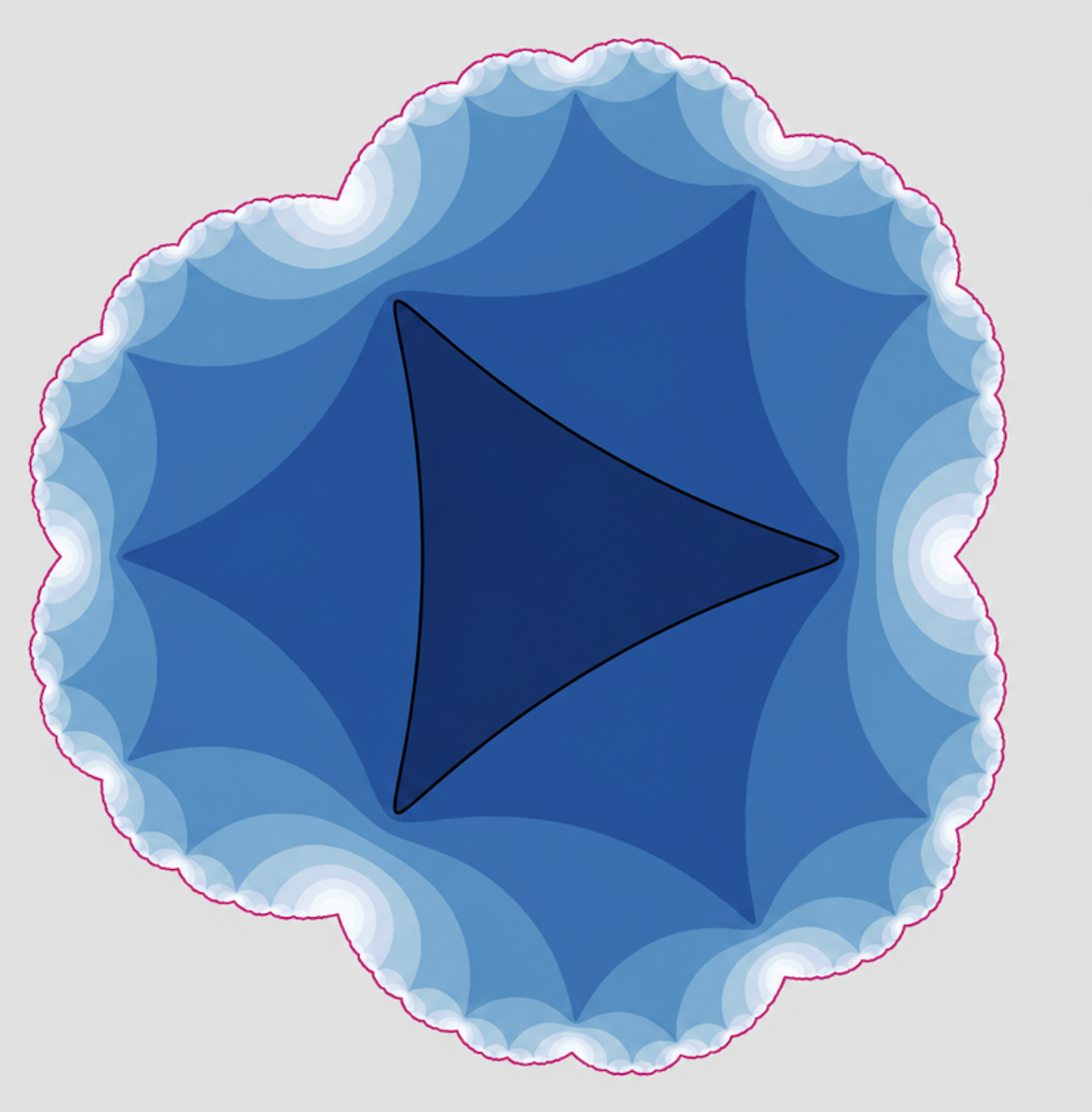}
        \caption{$\sigma: r(\mathbb{D}^*) \rightarrow \Chat$ }
        \label{fig:first}
    \end{subfigure}
    \hfill
    \begin{subfigure}[b]{0.32\textwidth}
        \centering
        \includegraphics[width=\linewidth]{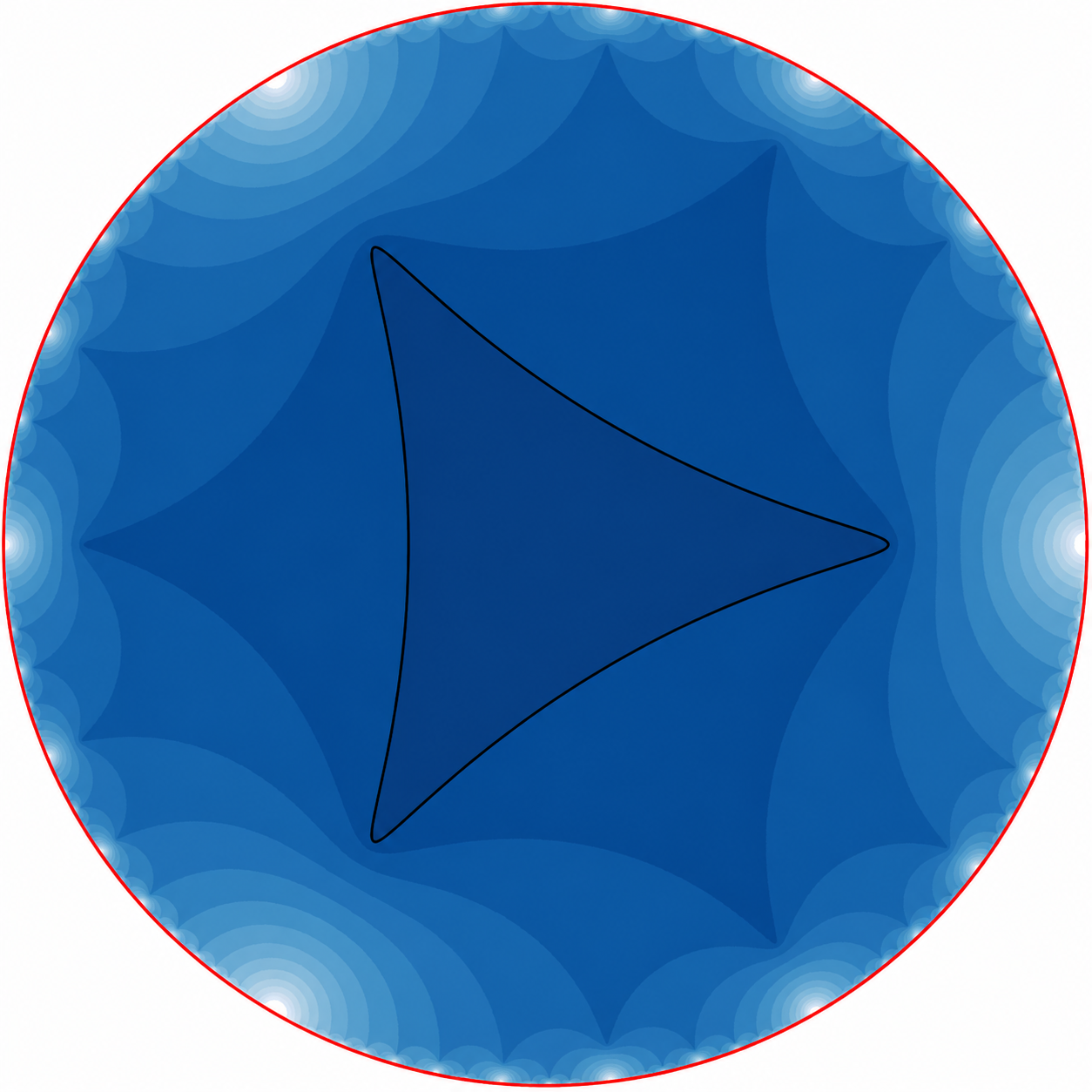}
        \caption{$\sigma_f: f(A_{\sqrt{s}}) \rightarrow \mathbb{D}$}
        \label{fig:second}
    \end{subfigure}
    \hfill
    \begin{subfigure}[b]{0.32\textwidth}
        \centering
        \includegraphics[width=\linewidth]{interiorunitdisc.png}
        \caption{$\overline{z}^2: \mathbb{D} \rightarrow \mathbb{D}$}
        \label{fig:third}
    \end{subfigure}
    \captionsetup{width=.95\textwidth}
    \caption{This Figure illustrates the description of (A) the hyperbolic Schwarz reflection $\sigma_c$ for $c>2$ as a mating of (B) the annular Schwarz reflection $\sigma_f$, with (C) the polynomial $\overline{z}^2$. We prove such a mating description in the broader context of hyperbolic Schwarz reflections with connected and full filled Julia set; see Figure \ref{fig:schwarz-reflection-sphere} for an example where $\interior K(\sigma)$ is disconnected.}
    \label{nonpinchedfig}
    \end{minipage}
\end{figure}
}

  \begin{example}\label{first_example} Let $r_c(z):=z+\frac{1}{cz^2}$. For $c\geq2$, the map $r_c$ is univalent (injective) on $\mathbb{D}^*:=\{z : |z|>1\}$, and $r_c(\mathbb{D}^*)$ is a quadrature domain with Schwarz Reflection given by
\begin{equation}\label{deltoidsystem1}  \sigma_c(z):=r_c\left(\frac{1}{\overline{r_c^{-1}(z)}}\right) \textrm{ for } z\in r_c(\mathbb{D}^*), \end{equation}
where the unique $\mathbb{D}^*$-valued branch of $r_c^{-1}$ is chosen in (\ref{deltoidsystem1}). The formula (\ref{deltoidsystem1}) defines an  \emph{open} dynamical system
\begin{equation}\label{dynsystem} \sigma_c: r_c(\mathbb{D}^*) \rightarrow \Chat,
\end{equation} 
where we use the terminology open to underscore that the range $\Chat$ of $\sigma_c$ properly contains the domain $r_c(\mathbb{D}^*)$, and so not all points have well-defined forward orbits. Indeed, $\Chat$ naturally decomposes into two $\sigma_c$-invariant regions: the \emph{escaping set} $T(\sigma_c)$ consisting of points without a well-defined forward orbit (colored blue in Figures \ref{fig:first1}, \ref{fig:first}), and its complement the \emph{non-escaping set} $K(\sigma_c)$ (colored grey in Figures \ref{fig:first1}, \ref{fig:first}).
\end{example}

In \cite{MR4706575}, the system $\sigma_2: r(\mathbb{D}^*) \rightarrow \Chat$ was described as a \emph{conformal mating} of two simpler systems as illustrated in Figure \ref{pinchedfig}: in particular they showed that the non-escaping dynamics $\sigma_2: K(\sigma_2)\rightarrow K(\sigma_2)$ are equivalent to the polynomial $\overline{z}^2: \mathbb{D}\rightarrow\mathbb{D}$, and the escaping dynamics $\sigma_2: T(\sigma_2)\rightarrow T(\sigma_2)$ are equivalent to the \emph{Nielsen map} $\rho: \mathbb{D}\setminus T\rightarrow\mathbb{D}$ defined piecewise by reflection in the arcs of a hyperbolic triangle $T$ (see Definition \ref{necklace} for a precise definition). Their description has given rise to a very active field of study in recent years: see for instance \cite{MR4365080}, \cite{MR4543152}, \cite{MR4472055}, \cite{MR4536467}, \cite{MR4588722}, \cite{MR4667419}, \cite{MR4806290}, \cite{MR4940720}, \cite{MR4961767}, \cite{MR5049498}, \cite{LUO_NTALAMPEKOS_2026}.

Interestingly the above mating description for $\sigma_2$ does not extend to $\sigma_c$ for $c>2$. One key distinction is that while $\sigma_c$ is hyperbolic for $c>2$ (uniformly expanding on a neighborhood of its Julia set), the map $\sigma_2$ has neutral fixed points on its Julia set. This naturally leads to the following question:

\vspace{2mm}

\noindent \emph{Is there a conformal mating description for Schwarz reflections in the hyperbolic setting, and if so, how do the hyperbolic and non-hyperbolic descriptions relate? }
 
 \vspace{2mm}
 

In this paper we demonstrate a conformal mating description in the hyperbolic setting. The key fact is that the escaping dynamics of hyperbolic Schwarz reflections admit an independent characterization by what we call \emph{Annular Schwarz Reflections}; we give a precise definition below in Definition \ref{sigmafdefn}. Like Nielsen maps, annular Schwarz reflections admit a concrete, finite-dimensional description, and they play the same role in the hyperbolic setting that Nielsen maps play in non-hyperbolic setting such as in \cite{MR4273178}, \cite{MR4706575}, \cite{luo2024generaldynamicaltheoryschwarz}, \cite{MR4961627}. In both hyperbolic and the above non-hyperbolic settings the non-escaping dynamics are described by polynomials.


Our results (stated precisely in Section \ref{main_resul_sec}) show that every hyperbolic Schwarz reflection determines two unique factors: a hyperbolic polynomial describing non-escaping dynamics, and an annular Schwarz reflection describing escaping dynamics (Theorem \ref{mating_thm}).  Conversely, any two such factors give rise to a unique hyperbolic Schwarz reflection (Theorem \ref{mating_thm2}). Lastly, we show that these annular Schwarz reflections $\sigma_f$ converge in a strong sense to the Nielsen maps $\rho$ as hyperbolicity degenerates (for instance as $c\rightarrow2$ in Example \ref{first_example}).

\begin{figure}[H]
    \centering
    \begin{subfigure}[b]{0.49\textwidth}
        \centering
        \includegraphics[width=\linewidth]
        {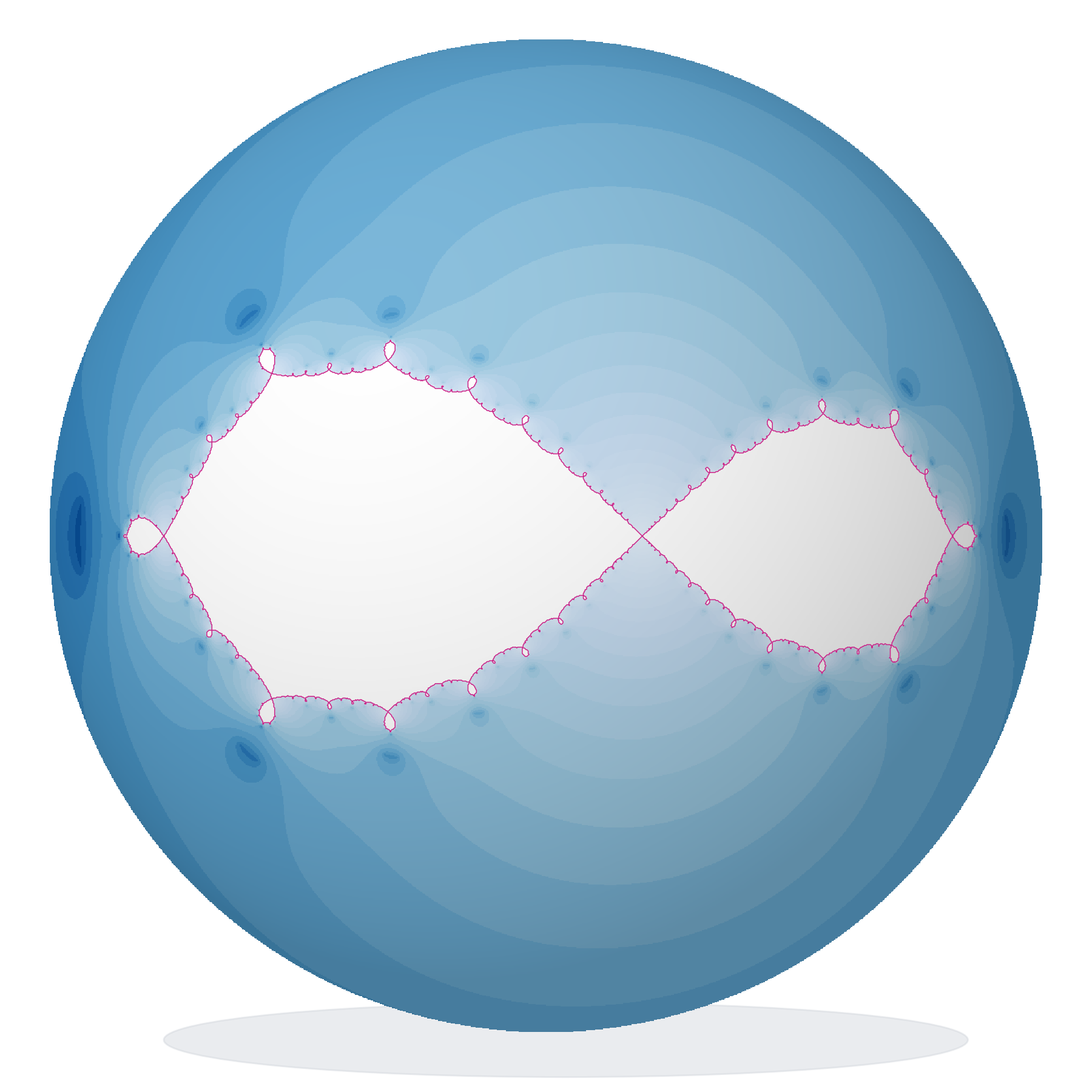}
        \caption{Front hemisphere.}
    \end{subfigure}
    \hfill
    \begin{subfigure}[b]{0.49\textwidth}
        \centering
        \includegraphics[width=\linewidth]
        {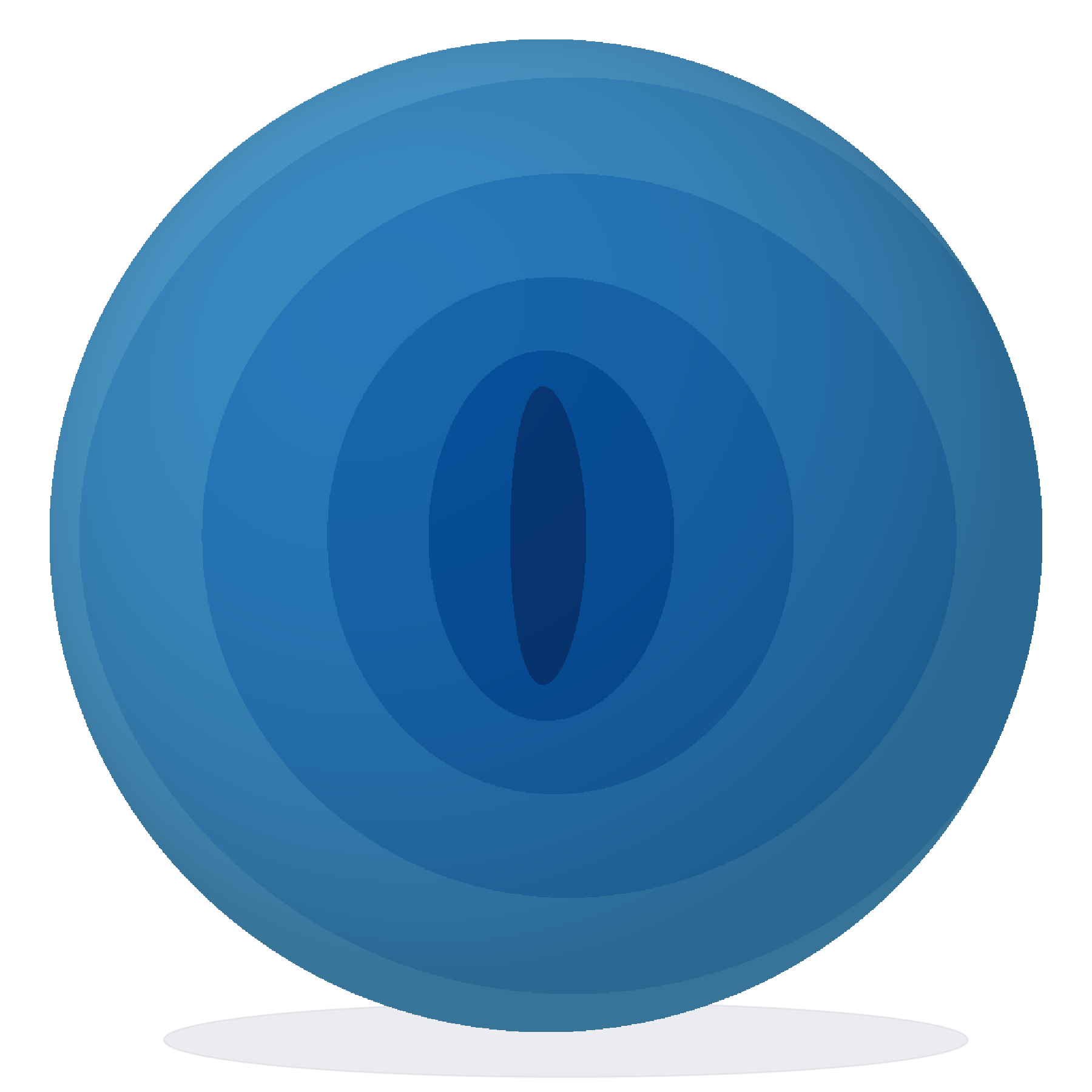}
        \caption{Opposite hemisphere.}
    \end{subfigure}    
    \captionsetup{width=\textwidth}
    \caption{Illustrated is the dynamical plane (illustrated on either hemisphere of $\Chat$) of a hyperbolic $\sigma=\sigma_r\in\Sigma_d$ with $r$ univalent in a neighborhood of $\mathbb{D}^*$ and so that $\sigma$ has an attracting $2$-cycle in $K(\sigma)$. Here the filled Julia set is in white and the escaping set is in blue. The fundamental tile $T^0(\sigma)$ is shown in dark blue on the opposite hemisphere.}
    \label{fig:schwarz-reflection-sphere}
\end{figure}

\subsection{Definitions and Notation.}\label{defnsec} In this Section \ref{defnsec} we will introduce some notation and definitions we will use throughout this paper before stating our results precisely in the next Section \ref{main_resul_sec}. 


We let $d\geq3$ and denote by $\rat_d$ the collection of degree $d$ rational maps $r: \Chat\rightarrow\Chat$.


\begin{notation} Let $r\in\rat_d$, and assume $r$ is univalent (injective) in a closed Euclidean disc $D$. Let $\rho: \Chat\rightarrow\Chat$ denote the reflection map in $\partial D$ (the unique anti-conformal self-map of $\Chat$ fixing $\partial D$ pointwise). The Schwarz Reflection for the quadrature domain $r(D)$ is given by 
\begin{equation}\label{deltoidsystem}  \sigma:= r\circ\rho\circ r^{-1} \textrm{ in } r(D), \end{equation}
where the unique $D$-valued branch $r^{-1}: r(D) \rightarrow D$ is chosen in (\ref{deltoidsystem}).
\end{notation}

\begin{rem} Note that the definition (\ref{deltoidsystem}) of $\sigma$ depends both on $r$ and a choice of the disc $D$; we use any of the notations $\sigma$, $\sigma_r$, $\sigma_{r, D}$ depending on which dependence we wish to emphasize.
\end{rem}

\begin{rem} The map $r\mapsto \sigma_r$ is not injective, as $\sigma_{r\circ M}=\sigma_{r}$ for any M\"obius transformation $M$ preserving $D$ (set-wise). 
\end{rem}

\noindent We now define a decomposition of the dynamical plane of $\sigma$ following Example \ref{first_example}.

\begin{definition}\label{tilingsetjuliasetdefn} We define the \emph{Tile} of $\sigma$ as $T^0(\sigma):=\closure(\Chat\setminus r(D))$, and the \emph{Tiling Set} by
\begin{equation} T(\sigma):= \{ z\in \Chat: \sigma^n(z) \in T^0(\sigma) \textrm{ for some } n\geq0 \}.
\end{equation}
We define the \emph{Filled Julia Set} of $\sigma$ by
\begin{equation} K(\sigma):= \Chat\setminus T(\sigma) = \{z \in \Chat : \sigma^n(z)\not\in T^0(\sigma) \textrm{ for all } n \geq0\},
\end{equation}
and the \emph{Julia set} of $\sigma$ by $\mathcal{J}(\sigma):= \partial K(\sigma)$.
\end{definition}

\begin{rem} That the terminology \emph{Filled Julia Set} and \emph{Julia Set} is appropriate will be justified in the following sections. We also remark that Definition \ref{tilingsetjuliasetdefn} is an instance of a broader definition in which one decomposes an open dynamical system into an escaping set (tiling set) and non-escaping set (filled Julia set), although our definition differs slightly as we include $\partial T^0(\sigma)$ in $T(\sigma)$ despite all forward-orbits being defined on $\partial T^0(\sigma)$. We also mention that in the literature the points $\{r(z) : r'(z)=0 \textrm{ and } z\in\partial D\}$ are removed from the tile by definition (and assigned to the Filled Julia set), but this will not be relevant to us as we will be primarily interested in the following class of $r$ without any critical points on $\partial D$.
\end{rem}

\begin{notation}\label{sigma_rdefn} We denote by $\Sigma_d$ the collection of $\sigma_{r, D}$ satisfying:
\begin{enumerate} \item $r\in\rat_d$ is univalent in a neighborhood of $D$, 
\item $K(\sigma)$ is connected and full (meaning $\Chat\setminus K(\sigma)$ is connected), and
\item The critical values of $\sigma$ contained in $K(\sigma)$ are attracted to attracting cycles. 
\end{enumerate}
\end{notation}

\begin{rem} We will see in Section \ref{dynam_plane_sec} that Condition (3) of Definition \ref{sigma_rdefn} ensures that $\sigma$ is hyperbolic: namely that $\sigma$ is uniformly expanding on a neighborhood of $\mathcal{J}(\sigma)$. The class $\Sigma_d$ is the main focus of this paper; we call this the class of \emph{hyperbolic Schwarz reflections}, although \emph{hyperbolic Schwarz reflections with connected and full filled Julia set} is a lengthier but more accurate terminology.
\end{rem}

As discussed in the Introduction, our first main result is that each $\sigma\in\Sigma_d$ is the conformal mating of a hyperbolic polynomial with an object we call an \emph{annular Schwarz reflection} which we will now describe (see Figure \ref{fig:second}).


\begin{notation} For $0<s<1$, we let $A_s:=\{z: s<|z|<1\}$. We denote by $\mathcal{F}_d$ the collection of maps $f$ satisfying:
\begin{enumerate} \item $f: A_s\mapsto\mathbb{D}$ is holomorphic and proper for some $0<s<1$, 
\item $\deg(f|_\mathbb{T})=1$, $\deg(f|_{s\mathbb{T}})=d-1$, and 
\item $f$ is univalent in a neighborhood of $\overline{A_{\sqrt{s}}}$.
\end{enumerate}
\end{notation}

\begin{rem} The space $\mathcal{F}_d$ is finite-dimensional and indeed is parametrized by the zeros of $f\in\mathcal{F}_d$ up to post-composition of $f$ with a rotation. By way of analogy, the proper maps $\mathbb{D}\mapsto\mathbb{D}$ are the Blaschke products and are also parametrized by zero-sets up to post-composition with a rotation. Maps in $\mathcal{F}_d$ even admit quite explicit formulas as certain infinite products of Blaschke factors and their quotients (see Theorem 1.3 of \cite{MR3716945}). We discuss more aspects of the family $\mathcal{F}_d$ in Section \ref{propofF}.
\end{rem}

\begin{definition}\label{sigmafdefn} Let $\rho_s(z):=s/\overline{z}$ be reflection in $\sqrt{s}\mathbb{T}$. For each $f\in\mathcal{F}_d$, we define the \emph{annular Schwarz reflection} $\sigma_f$ by 
\begin{equation}\label{fsystem} \sigma_f(z):= f\circ \rho_s \circ f^{-1}(z) \textrm{ for } z\in f\left(\overline{A_{\sqrt{s}}}\right), \end{equation}
where the unique $A_{\sqrt{s}}$-valued branch $f^{-1}: f\left(A_{\sqrt{s}}\right)\rightarrow A_{\sqrt{s}}$ is chosen in (\ref{fsystem}). We denote $\Sigma_d^\mathcal{F}:=\left\{ \sigma_f: f \in\mathcal{F}_d \textrm{ and } \sigma_f(1)=1 \right\}$ and $T^0(\sigma_f):=\mathbb{D}\setminus f\left(A_{\sqrt{s}}\right)$.
\end{definition}

\begin{rem} The normalization $\sigma_f(1)=1$ can always be obtained by post-composing $f$ with a rotation.
\end{rem}

\begin{rem} The definitions of annular Schwarz reflections $\sigma_f\in\Sigma_d^\mathcal{F}$ and Schwarz reflections $\sigma\in\Sigma_d$ are evidently similar; one key distinction is that $\sigma_f$ is $\overline{\mathbb{D}}$-valued whereas $\sigma$ is $\Chat$-valued. Another key distinction is that, as we will prove in Theorem \ref{tilingset=disc}, the natural decomposition of the dynamical plane of any annular Schwarz reflection is always trivial (see Figure \ref{fig:second}): the tiling set is always $=\mathbb{D}$, and so its complement (in $\overline{\mathbb{D}}$) is always $=\mathbb{T}$. 
\end{rem}

Our first main result is that each $\sigma\in\Sigma_d$ is a \emph{conformal mating} of a hyperbolic $p$ with $\sigma_f\in\Sigma_d^\mathcal{F}$, as in Definition \ref{conf_mat_defn} below. Roughly the definition requires that $\sigma\in\Sigma_d$ be simultaneously conjugate on $K(\sigma)$ to a hyperbolic polynomial $p: K(p) \rightarrow K(p)$ (Condition (A)) while also conjugate on $T(\sigma)$ to $\sigma_f: \mathbb{D}\setminus T^0(\sigma_f)\rightarrow\mathbb{D}$ (Condition (B)), and that these two conjugacies ``match up'' dynamically along $\mathcal{J}(\sigma)$ (Condition (C)). This is made precise in the following Definition \ref{conf_mat_defn} but is comparatively technical and can be skipped on first reading; we refer to \cite{MR3088260} for a broader discussion of conformal mating. Our Definition closely follows Definition 2.42 of \cite{MR4381220} and Section 10.2.1 of \cite{MR4961627}. 



\begin{notation} We denote by $\pold$ the collection of monic polynomials of degree $d$ in one complex variable $z$, and by $\overline{\pold}:=\{\overline{p(z)}: p\in\pold\}.$ We denote by $\overline{\textrm{Pol}_\textrm{d}^\textrm{HC}}$ the hyperbolic $p\in\overline{\textrm{Pol}_\textrm{d}}$ with connected filled Julia set.
\end{notation}

\begin{rem}\label{mathcalErem} The maps $\overline{z}^{d-1}: \mathbb{T}\rightarrow \mathbb{T}$ and $\sigma_f: \mathbb{T}\rightarrow\mathbb{T}$ are always conjugate for any $\sigma_f\in\Sigma_d^\mathcal{F}$, and we denote by $\mathcal{E}_f$ the unique orientation-preserving homeomorphism $\mathcal{E}_f: \mathbb{T}\rightarrow\mathbb{T}$ so that
\[ \mathcal{E}_f\circ\sigma_f=\overline{\mathcal{E}_f}^{d-1} \textrm{ on } \mathbb{T}, \textrm{ and } \mathcal{E}_f(1)=1. \] 
We denote the B\"ottcher coordinate for $p\in\overline{\textrm{Pol}_\textrm{d}^\textrm{HC}}$ by $\beta: \mathbb{D}^* \rightarrow \Chat\setminus K(p)$, and we normalize $\beta$ so that $\beta'(\infty)>0$.
\end{rem}

\begin{definition}\label{conf_mat_defn} Let $\sigma\in\Sigma_d$, $\sigma_f\in\Sigma_d^\mathcal{F}$, and  $p\in\overline{\poldmh}$.
 We say that $\sigma$ is a \emph{conformal mating} of $p$ and $\sigma_f$ if there exists a
 \begin{enumerate}
 \item quasiconformal homeomorphism $\psi_p$ of a neighborhood $U\supset K(\sigma)$ onto a neighborhood $V\supset K(p)$ with $\partial_{\overline{z}}\psi_p=0$ a.e. on $K(\sigma)$ and $\psi_p(\mathcal{J}(\sigma))=\mathcal{J}(p)$,
 \item a conformal map $\psi_f: T(\sigma) \rightarrow \mathbb{D}$ with $\psi_f(T^0(\sigma))=T^0(\sigma_f)$,
 \end{enumerate}
 so that
 \begin{enumerate}[label=(\Alph*)]
\item $\psi_p\circ \sigma = p\circ \psi_p$ on a neighborhood $W$ of $K(\sigma)$ where $W$, $\sigma(W)\subset U$, 
\item $\psi_f\circ \sigma = \sigma_f \circ \psi_f$ on $T(\sigma)\setminus T^0(\sigma)$,
\item $\psi_p\circ\psi_f^{-1}=\beta\circ\overline{\mathcal{E}_f}$ on $\mathbb{T}$. 
\end{enumerate}
\end{definition}


\vspace{2.5mm}

\subsection{Statement of Main Results}\label{main_resul_sec}

With the Definitions and Notation of Section \ref{defnsec} in hand, we may now state our main results more precisely. 

\begin{thmx}\label{mating_thm} For each $\sigma\in\Sigma_d$, there exist $p\in\overline{\poldmh}$ and $\sigma_f\in\Sigma_d^\mathcal{F}$ so that $\sigma$ is the conformal mating of $p$ with $\sigma_{f}$, and such $p$, $\sigma_f$ are unique up to M\"obius conjugacy.
\end{thmx}

\noindent Conversely, we have the following.

\begin{thmx}\label{mating_thm2} For any pair $p\in\overline{\poldmh}$, $\sigma_f\in\Sigma_d^\mathcal{F}$, there exists $\sigma\in\Sigma_d$ so that $\sigma$ is the conformal mating of $p$, $\sigma_f$, and moreover such $\sigma$ is unique up to M\"obius conjugacy.
\end{thmx}


As already discussed, there already exists an analogous mating description for many classes of non-hyperbolic Schwarz reflections $\sigma:= r\circ\rho\circ r^{-1}$ having neutral fixed points on $\mathcal{J}(\sigma)$.  Our other main result, Theorem \ref{reconciliation_thm}, is meant to consolidate the descriptions in these two settings. We first recall the setup from \cite{MR4706575}, \cite{MR4961627}. The following definition of a \emph{Nielsen map} is illustrated in Figure \ref{fig:second1} for $d=3$.



\begin{definition}\label{necklace} Suppose $C_1$, ..., $C_d\subset\mathbb{C}$ are Euclidean circles, and denote by $\interior(C_i)$ the bounded component of $\mathbb{C}\setminus C_i$. Assume that:
\begin{enumerate} \item Each $C_i$ is tangent to $C_{i+1}$ (with $i+1$ taken mod $d$),
\item Each circle $C_i$ intersects $\mathbb{T}$ at two points and at right angles, and
\item  $\interior(C_i)\cap\interior(C_j)=\emptyset$ for each $i\not=j$.
\end{enumerate}
Let $r_i$ denote reflection in $C_i$, and $T:=\mathbb{D}\setminus\cup_{i=1}^d (C_i\cup\interior(C_i))$. We define the \emph{Nielsen map} $\rho: \overline{\mathbb{D}}\setminus T \rightarrow \overline{\mathbb{D}} $ by $\rho(z) = r_i(z)$ for $z\in C_i\cup\interior(C_i)$. We denote $T^0(\rho):=T$.
\end{definition}

In various settings, the works \cite{MR4706575}, \cite{MR4381220} and \cite{MR4961627} studied the description of non-hyperbolic $\sigma_r$ as conformal matings of hyperbolic polynomials with Nielsen maps $\rho$; we refer to these works for precise statements. By perturbing $r$ we may always obtain a hyperbolic Schwarz reflection; how does the mating change under this perturbation? The answer is given in the following Theorem and illustrated in Figures \ref{fig:second1}, \ref{fig:second}.


\begin{thmx}\label{reconciliation_thm} Assume $\sigma_r$ is the mating of a hyperbolic polynomial and a Nielsen map $\rho$, with $\interior T^0(\sigma_r)$ connected. Let $(r_n)_{n=1}^\infty\subset\ratd$ converge uniformly on $\Chat$ to $r$, and assume $\sigma_n:=\sigma_{r_n}\in\Sigma_d$ for all $n$. Then $T(\sigma_{n})\rightarrow T(\sigma)$ in the Carath\'eodory sense, and each $\sigma_n$ is the mating of a hyperbolic polynomial with $\sigma_{f_n}\in \Sigma_d^\mathcal{F}$ where $\sigma_{f_n}\rightarrow\rho$ uniformly on compact subsets of $\mathbb{D}\setminus T^0(\rho)$.
\end{thmx}

The rest of the paper is organized as follows. Section \ref{dynam_plane_sec} develops some basic properties of the dynamical plane of $\sigma\in\Sigma_d$. Section \ref{rtobf} is devoted to the proof of Theorem \ref{mating_thm}, and after proving some needed results about the family $\mathcal{F}$ and $\Sigma_d^\mathcal{F}$ in Section \ref{propofF}, we prove Theorem \ref{mating_thm2} in Section \ref{bftor}. In Section \ref{degen_sec} we prove Theorem \ref{reconciliation_thm}. The proofs of Theorems, \ref{mating_thm}  \ref{mating_thm2} use standard quasiconformal surgeries, and the proof of Theorem \ref{reconciliation_thm} uses Carath\'eodory's kernel convergence theorem.

\vspace{5mm}

\noindent\textbf{Acknowledgements.} The main results and proofs of this manuscript were developed by the author without the use of AI. The author used AI to referee an earlier draft of this manuscript which led to various improvements. The figures were also generated with the use of AI.

\section{The dynamical plane of $\sigma$}\label{dynam_plane_sec} In this Section we establish some basic properties of the (open) dynamical systems $\sigma\in\Sigma_d$. 

\begin{notation} Let $\sigma\in\Sigma_d$. We denote $T^0:=T^0(\sigma)$ and let $A_n:=\sigma^{-n}(T^0)$ for $n\geq0$.
\end{notation}

\begin{prop}\label{dc_exist} There is a doubly-connected component of $A_1$ whose boundary contains $\partial A_0$, and if there are any other components of $A_1$ they must be simply-connected.
\end{prop}

\begin{proof} Since $\sigma$ fixes point-wise $\partial A_0$ and is anti-holomorphic, it follows that points in a one-sided neighborhood of $\partial A_0$ are mapped by $\sigma$ into $A_0$. Thus there is a component $A$ of $A_1$ with $\partial A_0\subset \partial A$ and since $A\not=\Chat$, $A$ must have at least two complementary components. If $A$ were to have a third complementary component, this would contradict the assumption (see Notation \ref{sigma_rdefn}) that $K(\sigma)$ is connected as the two complementary components $\not=A_0$ would both contain points in $K(\sigma)$ (obtained as the intersection of a sequence of nested compact sets). Similarly, any other component $A$ of $A_1$ must be simply connected, since if there were two components in $\Chat\setminus A$, these two components would disconnect $K(\sigma)$. 
\end{proof}

\begin{prop}\label{dc_exist2} The map $\sigma: A_1\rightarrow A_0$ is a proper branched cover of degree $d$, and for $n\geq2$, $\sigma: A_n \rightarrow A_{n-1}$ is a proper branched cover of degree $d-1$. 
\end{prop}

\begin{proof} Recall that $\sigma:=r\circ\rho\circ r^{-1}$ in $r(D)$, and $r^{-1}$, $\rho$ are both $1:1$ in $r(D)$, $D$ (respectively). Since $r: r^{-1}(A_0) \rightarrow A_0$ is a proper degree $d$ branched cover with $r^{-1}(A_0)\subset \Chat\setminus D$, it follows that $\sigma: A_1\rightarrow A_0$ is also a proper degree $d$ branched cover.

 Similarly the second statement follows since each point in $A_n$ for $n>1$ has $d-1$ preimages under $r$ in $\Chat\setminus D$. 
\end{proof}


We will now take $\sigma$ (which we recall is not defined in $A_0$) and ``glue in'' an attracting fixed point into $A_0$ to produce a rational function. The construction, illustrated in Figure \ref{fig:sigmaqrconstr}, is that of straightening a rational-like mapping to produce a genuine rational map which is hybrid equivalent to the given rational-like mapping. 

\begin{figure}[htbp]
    \centering
    \includegraphics[width=\textwidth]{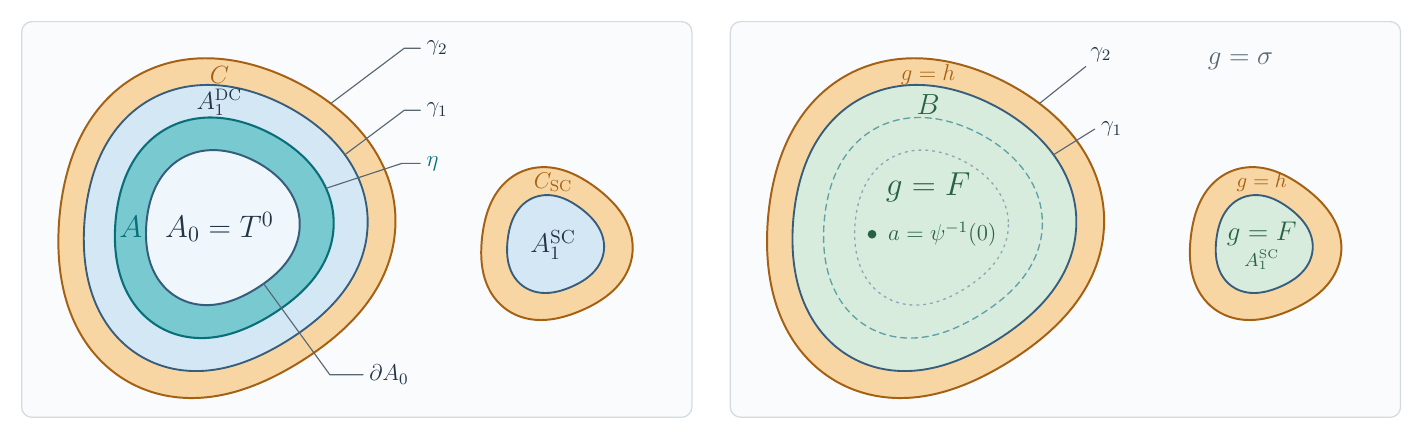}
    \caption{Illustrated is the notation of Construction \ref{sigmaqrconstr}.}
    \label{fig:sigmaqrconstr}
\end{figure}


\begin{construction}\label{sigmaqrconstr} Let $\sigma\in\Sigma_d$. Assume for now that $\CV(\sigma)\cap\partial A_0=\emptyset$ (see Remark \ref{cvassumpti} below). Let $A$ be a topological annulus with:
\begin{enumerate}
\item $A\subset A_1$, 
\item $\partial A_0\subset \partial A$, and 
\item $A\cap\CV(\sigma)=\emptyset$. 
\end{enumerate}
Denote by $\eta$ the boundary component of $A$ with $\eta\not=\partial A_0$. Each component of $\sigma^{-1}(A)$ is a topological annulus one of whose boundary components coincides with a boundary component of $A_1$. Let $C$ denote the component of $\sigma^{-1}(A)$ sharing a boundary component with the doubly-connected component $A_1^{\textrm{DC}}$ of $A_1$ (see Proposition \ref{dc_exist}). Denote $\gamma_1:=\partial C\cap \partial A_1^{\textrm{DC}}$, and denote the other component of $\partial C$ by $\gamma_2$. Let $B:=A_1^{\textrm{DC}}\cup A_0$ so that, by Proposition \ref{dc_exist}, $B$ is simply-connected. There exists a conformal map $\psi: B \rightarrow (1/2)\mathbb{D}$. Let $\ell\in\{1, ..., d-1\}$ be the degree of $\sigma: \gamma_1\rightarrow\partial A_0$. Define 
\[ F(z):=\psi^{-1}\left(\frac{\overline{\psi(z)}^{\ell}}{2}\right) \textrm{ for } z\in B.\]
 We have that: 
\begin{enumerate}
\item $F: \gamma_1 \rightarrow \psi^{-1}(\mathbb{T}/2^{\ell+1}) \subset B$,
\item $\sigma: \gamma_2 \rightarrow \eta \subset \partial A$
\end{enumerate}
are covering maps of the same degree, and thus there exists a quasiregular interpolation $h$ in $C$ between the maps (1) and (2) (see for instance Section 2.3 of \cite{MR3445628}), in other words there is a degree $\ell$ covering map $h$ of $C$ onto the annulus bounded by $\eta$ and $\psi^{-1}(\mathbb{T}/2^{\ell+1})$ so that $h(z)=F(z)$ on $\gamma_1$ and $h(z)=\sigma(z)$ on $\gamma_2$.

Consider any other component $C_{\textrm{SC}}$ of $\sigma^{-1}(A)$, which necessarily shares a boundary component with a simply-connected component $A_1^{\textrm{SC}}$ of $A_1$ by Proposition \ref{dc_exist}. Denote the other boundary component of $C_{\textrm{SC}}$ by $\eta$. We consider a conformal map $\phi: A_{0}\rightarrow \psi^{-1}((1/2)^{\ell+1}\mathbb{D})$, set $F:=\phi\circ\sigma$ in $A_1^{\textrm{SC}}$, and denote by $h$ the annular interpolation in $C_{\textrm{SC}}$ between $F|_{\partial A_1^{\textrm{SC}}}$ and $\sigma|_\eta$.  We consider
\begin{equation}\label{defn_of_g1}
g:=\begin{cases} F \textrm{ in } A_1\cup A_0,  \\  h \textrm{ in } \sigma^{-1}(A),  \\ \sigma \textrm{ otherwise. } \end{cases}
\end{equation}
\end{construction}

\begin{rem}\label{cvassumpti} We assumed in Construction \ref{sigmaqrconstr} that $\CV(\sigma)\cap\partial A_0=\emptyset$. If $\CV(\sigma)\cap\partial A_0\not=\emptyset$, we choose instead $A$ so that the inner boundary does not contain any critical values of $\sigma$ and proceed as in Construction \ref{sigmaqrconstr}.
\end{rem}


\begin{prop}\label{gisqr1} The map $g: \Chat\rightarrow\Chat$ is quasiregular. 
\end{prop}

\begin{proof} The map $g$ is quasiregular in each of the regions in the piecewise definition, and continuous across the common boundaries since $h$ interpolates between $F$ and $\sigma$. Hence by removability of analytic arcs for quasiregular mappings, $g$ is quasiregular on $\Chat$. 
\end{proof}

\begin{definition}\label{gisqrsectbelt} We define a Beltrami coefficient $\mu$ in $\Chat$ as follows:
\begin{enumerate}
\item $\mu|_{A_1\cup A_0}:=0$,
\item $\mu|_{\sigma^{-1}(A)}:=h^*(\mu|_{A_1\cup A_0})=h^*(0)$ (the pull back of $\mu|_{A_1\cup A_0}$ under $h$ defines $\mu$ in $\sigma^{-1}(A)$),
\item $\mu|_{\sigma^{-n}(A)}:=\sigma^*(\mu|_{\sigma^{-(n-1)}(A)})$ for $n\geq2$ (the pull back of $\mu|_{\sigma^{-(n-1)}(A)}$ under $\sigma$ defines $\mu$ in $A_n$), and  
\item $\mu:=0$ elsewhere.
\end{enumerate}
\end{definition}

\begin{prop}\label{g_invariance} The Beltrami coefficient $\mu$ is $g$-invariant, and $||\mu||_{L^\infty}<1$. 
\end{prop}

\begin{proof} The Beltrami coefficient $\mu$ is $g$-invariant in $T(\sigma)$ by definition of $\sigma$ and (\ref{defn_of_g1}). In $\Chat\setminus T(\sigma)$ we have $\mu=0$ and $g=\sigma$ is anti-holomorphic so $\mu$ is $g$-invariant in all of $\Chat$. Moreover, since $h$ is quasiregular, we have $||\mu||_{L^\infty(A_2)}<1$, and so since $\sigma$ is anti-holomorphic, we conclude $||\mu||_{L^\infty(\Chat)}=||\mu||_{L^\infty(A_2)}<1$.
\end{proof}

\begin{notation}\label{phifirstinr} By existence of solutions to the Beltrami equation, there exists a quasiconformal mapping we denote $\phi: \Chat\rightarrow\Chat$ so that $\phi_{\overline{z}}=\mu\phi_z$, and we normalize $\phi$ so that $\phi\circ\psi^{-1}(0)=\infty$ and $\phi(\infty)=\psi^{-1}(0)$. We set 
\[   r:=  \phi \circ g \circ \phi^{-1}. \] 
\end{notation}

\begin{thm}\label{r_prop} The map $r:\Chat\rightarrow\Chat$ is a degree $d-1$ rational map, and $\infty$ is an attracting fixed point of $r$.
\end{thm}

\begin{proof} That $r$ is anti-holomorphic in $\Chat$ follows from $g$-invariance of $\mu$ and $\phi_{\overline{z}}=\mu\phi_z$. Hence $r$ is a rational map. That $\infty$ is attracting follows from the definition of $g$ and the normalization of $\phi$, and the degree is readily checked. 
\end{proof}

\begin{notation} Let $\mathcal{A}(\infty)$ denote the basin of attraction for $\infty$ for $r$.
\end{notation}

\begin{thm}\label{respectivecount} The rational map $r$ is hyperbolic with connected and full filled Julia set, and $\phi$ is a conjugacy between the filled Julia sets of $r$ and $\sigma$ with $\phi_{\overline{z}}=0$ a.e. on $K(\sigma)$. Moreover, $\phi(T(\sigma))=\mathcal{A}(\infty)$, and $\phi(\mathcal{J}(\sigma))=\mathcal{J}(r)$. 
\end{thm}

\begin{proof} The identity $\phi(T(\sigma))\subset\mathcal{A}(\infty)$ is readily verified, as is $\phi(K(\sigma))\subset K(r)$ and hence $\phi(\mathcal{J}(\sigma))\subset \mathcal{J}(r)$. Since these sets partition the plane, we in fact have $=$ in all three relations. 

The map $\phi$ is conformal in $\interior K(\sigma)$ since $\mu=0$ a.e. on $K(\sigma)$, and so $\phi$ maps attracting cycles of $\sigma$ to attracting cycles of $r$. Moreover, the critical values of $r$ consist of:
\begin{enumerate}
\item $\infty$, 
\item the $\phi$-images of the critical values of $r$ in the tiling set, and
\item  the $\phi$-images of those critical values of $\sigma$ which are attracted to attracting cycles within $K(\sigma)$.
\end{enumerate}
Hence all finite critical values of $r$ lie in basins of attraction for attracting cycles, and so $r$ is hyperbolic. 

Note that since $K(\sigma)=\phi^{-1}(K(r))$ and $K(\sigma)$ is connected and full, we deduce that $K(r)$ is connected and full. Thus $\mathcal{A}(\infty)$ is connected and full and so $\mathcal{J}(r)$ is connected. 
\end{proof}

\begin{cor}\label{tilingsetopen} The tiling set $T(\sigma)$ is open and simply connected.
\end{cor}

\begin{proof} This follows since $T(\sigma)=\phi^{-1}(\mathcal{A}(\infty))$ and $\mathcal{A}(\infty)$ is open and simply connected by Theorem \ref{respectivecount}. 
\end{proof} 

\noindent Similarly we have the following consequence of Theorem \ref{respectivecount}. 

\begin{cor}\label{Jordan_curve_prop} For each $\sigma\in\Sigma_d$, one has $\mathcal{J}(\sigma)$ is closed and
 \[   \mathcal{J}(\sigma):=\partial K(\sigma)=\partial T(\sigma). \]
\end{cor}

\noindent The following Proposition establishes hyperbolicity of $\sigma$.

\begin{prop}\label{hyperbolicity} The map $\sigma$ is expanding on $\mathcal{J}(\sigma)$: namely there exists a conformal metric $\mu$ defined in a neighborhood $U$ of $\mathcal{J}(\sigma)$ so that $\inf_{z\in U}||D\sigma(z)||_\mu>1$.
\end{prop}

\begin{proof} There is a positive distance between the postcritical set of $\sigma$ and $\mathcal{J}(\sigma)$ by Theorem \ref{respectivecount}, and so the result follows from Theorem 19.1 of \cite{MilnorCDBook}. 
\end{proof}

\begin{cor} For each $\sigma\in\Sigma_d$, the Julia set of $\sigma$ is locally connected.
\end{cor}

\begin{proof} This follows from the assertion of Theorem \ref{respectivecount} that $\phi$ maps $\mathcal{J}(\sigma)$ homeomorphically onto $\mathcal{J}(r)$, and the fact that $\mathcal{J}(r)$ is locally connected since $r$ is hyperbolic with connected filled Julia set (see Chapter 19 of \cite{MilnorCDBook}). 
\end{proof}

\begin{thm}\label{p_prop} There exists $p\in\overline{\poldmh}$ and a quasiconformal homeomorphism $\psi_p: \Chat\rightarrow\Chat$ with $\partial_{\overline{z}}\psi_p=0$ a.e. on $\interior K(\sigma)$ with $\psi_p(K(\sigma))=K(p)$ so that $\psi_p\circ \sigma = p\circ \psi_p$ on a neighborhood of $K(\sigma)$. 
\end{thm}

\begin{proof} Since $\mathcal{A}(\infty)$ (the basin of attraction of $\infty$ for $r$) is open and simply-connected, the map $r: \mathcal{A}(\infty) \rightarrow \mathcal{A}(\infty)$ is conjugate to a Blaschke product. A standard quasiconformal surgery (see Chapter 3 of \cite{MR3445628}) shows that there is a quasiconformal conjugacy between $r|_{K(r)}$ and $p|_{K(p)}$ for some $p\in\overline{\poldmh}$ so that the Beltrami coefficient of the conjugacy vanishes on $K(r)$. Composing this conjugacy with $\phi$ from Notation \ref{phifirstinr} gives the desired conjugacy.
\end{proof}

\section{Decomposing $\sigma$ into $\sigma_f$, $p$}\label{rtobf}

Theorem \ref{p_prop} showed that $\sigma: K(\sigma) \rightarrow K(\sigma)$ is conjugate to $p: K(p)\rightarrow K(p)$, and so we turn to proving the corresponding conjugacy statement for $\sigma: \mathbb{D}\setminus T^0\rightarrow\mathbb{D}$. Recall we proved $T(\sigma)$ is open and simply connected in Corollary \ref{tilingsetopen}, and hence there exists a conformal map of $T(\sigma)$ onto $\mathbb{D}$.

\begin{notation}\label{nudefnhere} Let $\sigma\in\Sigma_d$ and $\psi: T(\sigma) \rightarrow \mathbb{D}$ a Riemann map. Consider the mapping 
\[ \psi\circ\sigma\circ\psi^{-1} : \mathbb{D}\setminus\psi(T^0)\rightarrow\mathbb{D}.  \]
\end{notation}

\begin{thm}\label{exist_of_f} There exists $\sigma_f\in\Sigma_d^\mathcal{F}$ so that $\psi\circ\sigma\circ\psi^{-1} =\sigma_f$.
\end{thm}

\begin{proof} Recall that $\sigma=\sigma_{r, D}$ where $r$ is rational and univalent on a neighborhood of $D$. Consider the topological annulus $A$ with outer boundary $\mathbb{T}$, and inner boundary $\psi\circ r(\partial D)$. Let $0<s<1$ be so that there exists a Riemann map $f: A_{\sqrt{s}} \rightarrow A$. Extend $f$ to $A_s$ by Schwarz Reflection: 
\begin{equation}\label{fsigmafdefn}    f(z):=  \psi\circ\sigma\circ\psi^{-1}\circ f\left(\frac{s}{\overline{z}}\right) \textrm{ for } s < |z| < \sqrt{s}.  \end{equation}

Note that $\psi^{-1}$ may not admit a single-valued extension to $\mathbb{T}$, and so it is not immediate that the formula (\ref{fsigmafdefn}) extends to a well-defined function on $|z|=s$. Nevertheless, since $\mathcal{J}(\sigma)=\partial T(\sigma)$ by Corollary \ref{Jordan_curve_prop}, and $\sigma(\mathcal{J}(\sigma))=\mathcal{J}(\sigma)$, we conclude from (\ref{fsigmafdefn}) that $f: A_s\rightarrow \mathbb{D}$ is proper. Hence $f: A_s\rightarrow \mathbb{D}$ extends continuously to a single-valued function on $s\mathbb{T}$. Moreover, since $\sigma: \mathcal{J}(\sigma) \rightarrow \mathcal{J}(\sigma)$ is of degree $d-1$, the map $f|_{s\mathbb{T}}$ is also of degree $d-1$. Post-composing $\psi$ with a rotation if necessary, we have $\sigma_f(1)=1$. As $f$ is univalent in a neighborhood of $\overline{A_{\sqrt{s}}}$, we conclude $f\in\mathcal{F}_d$. Moreover, since $\psi\circ\sigma\circ\psi^{-1}$ and $\sigma_f$ are both the identity on $\psi\circ r(\mathbb{T})$, we conclude that $\psi\circ\sigma\circ\psi^{-1}\equiv\sigma_f$.
\end{proof}

\noindent  We are now ready to collect our results and finish the proof of Theorem \ref{mating_thm}.

\vspace{2mm}


\noindent \emph{Proof of Theorem \ref{mating_thm}.} The polynomial $p$ was produced in Theorem \ref{p_prop} and $\sigma_f$ was produced in Theorem \ref{exist_of_f}, and the conformal mating Conditions (A) and (B) (see Definition \ref{conf_mat_defn}) were verified in Theorem \ref{p_prop} and Theorem \ref{exist_of_f}. Let us verify condition (C). Denote by $\psi_p: \Chat\rightarrow\Chat$ the quasiconformal mapping which restricts to a conjugacy $\psi_p: K(\sigma) \rightarrow K(p)$, and consider the composition
\[ \beta^{-1}\circ \psi_p \circ\psi_f^{-1} \]
defined in $\mathbb{D}$ extends to an orientation-reversing conjugacy $\mathbb{T}\rightarrow\mathbb{T}$ between $\sigma_f$ and $\overline{z}^{d-1}$, and after post-composing $\psi_f$ with a rotation if necessary this conjugacy fixes $1$ so that 
\[ \beta^{-1}\circ \psi_p \circ\psi_f^{-1} = \overline{\mathcal{E}_f},\]
(see Remark \ref{mathcalErem}) and so Condition (C) has been proven.

The uniqueness statement for $\sigma_f$ is readily verified. Uniqueness for $p$ follows from the fact that if $p$, $q$ are two hyperbolic polynomials which are hybrid conjugate in a neighborhood of their filled Julia sets, then $p$, $q$ are affine conjugate: see Corollary 2 to Proposition 6 in \cite{MR816367}. \qed


\section{Properties of $\mathcal{F}$ and $\Sigma_d^\mathcal{F}$.}\label{propofF}

In the previous Section \ref{rtobf} we decomposed $\sigma\in\Sigma_d$ into $\sigma_f$, $p$, and in the next Section \ref{bftor} we will do the opposite: we will weld $\sigma_f$, $p$ to produce $\sigma\in\Sigma_d$. But first in this Section \ref{propofF} we will need to record some properties of annular Schwarz reflections $\sigma_f\in\Sigma^\mathcal{F}_d$ which will be needed in Section \ref{bftor}. 

\begin{notation}\label{sigmafnotation} Let $f\in\mathcal{F}_d$. We denote $T^0:=\mathbb{D}\setminus f(A_{\sqrt{s}})$,  $A_n:=\sigma_f^{-n}(T^0)$ for $n\geq0$, and $T(\sigma_f)=\cup_{n=0}^\infty A_n$. 
\end{notation}

In Section \ref{rtobf} we proved some basic mapping properties of $\sigma\in\Sigma_d$ in the Tiling set $T(\sigma)$; now we will need to prove some corresponding results for $\sigma_f\in\Sigma_d^\mathcal{F}$. The statements of the following two Propositions \ref{annulus_structuref}, \ref{degreesigmafstatement} mimic those of Propositions \ref{dc_exist}, \ref{dc_exist2}.

\begin{prop}\label{annulus_structuref} There is a doubly-connected component of $A_1$ whose boundary contains $\partial A_0$, and if there are any other components of $A_1$ they must be simply-connected.
\end{prop}

\begin{proof} Since $\sigma_f$ is orientation-reversing and fixes $\partial A_0$ pointwise, there is a one-sided neighborhood of $A_0$ mapped into $A_0$ by $\sigma_f$, so there is a component $A_1^{\textrm{DC}}$ of $A_1$ with at least two complementary components. There can not be a third component in $\Chat\setminus A_1^{\textrm{DC}}$ since the proper map $\sigma_f$ would necessarily take on values in $\mathbb{D}^*$ in this third component, whereas $\sigma_f$ is $\overline{\mathbb{D}}$-valued. Thus $A_1^{\textrm{DC}}$ has exactly two complementary components, and the argument that any other components of $A_1$ are simply-connected is similar. 
\end{proof}

\begin{prop}\label{degreesigmafstatement} The map $\sigma_f: A_1\rightarrow A_0$ is a proper branched cover of degree $d$, and for $n\geq2$, $\sigma_f: A_n \rightarrow A_{n-1}$ is a proper branched cover of degree $d-1$. 
\end{prop}

\begin{proof} The proof is the same as that of Proposition \ref{dc_exist2}.
\end{proof}

We will now prove that for any $\sigma_f\in\Sigma_d^\mathcal{F}$ we have that $\partial T(\sigma_f)=\mathbb{T}$; in other words the partition of the dynamical plane of annular Schwarz reflections is always trivial. We do this somewhat similarly to the methods Section \ref{dynam_plane_sec}: we will glue in an attracting fixed point into the domain $T(\sigma_f)$ where $\sigma_f$ is not defined in order to obtain a Blaschke product, and we will deduce the desired dynamical properties of the annular Schwarz reflection from analagous dynamical properties of the Blaschke product. The following Construction \ref{sigmafgluing} mimics Construction \ref{sigmaqrconstr}. 

\begin{construction}\label{sigmafgluing} Let $\sigma_f\in\Sigma_d^\mathcal{F}$. Assume $\CV(\sigma_f)\cap\partial A_0=\emptyset$ (see Remark \ref{cvassumpti}). Let $A$ be a topological annulus with:
\begin{enumerate}
\item $A\subset A_1$, 
\item $\partial A_0\subset \partial A$, and 
\item $A\cap\CV(\sigma_f)=\emptyset$. 
\end{enumerate}
Each component of $\sigma_f^{-1}(A)$ is a topological annulus one of whose boundary components coincides with a boundary component of $A_1$. As in Construction \ref{sigmaqrconstr}, we glue in an attracting fixed point $z_0\in A_0$, and use each annular component of $\sigma_f^{-1}(A)$ to interpolate with $\sigma_f$; let $F$ denote the map in $A_1$, $h$ the interpolation and consider:
\begin{equation}\label{defn_of_g}
g:=\begin{cases} F \textrm{ in } A_1\cup A_0,  \\  h \textrm{ in } \sigma_f^{-1}(A),  \\  \sigma_f \textrm{ otherwise in } \mathbb{D}. \end{cases}
\end{equation}
and, after noting $g(\mathbb{T})=\sigma_f(\mathbb{T})=\mathbb{T}$, we extend $g$ to a mapping $g:\Chat\rightarrow\Chat$ by setting $g(z)=1/\overline{g(1/\overline{z})}$ for $z\in\mathbb{D}^*$.
\end{construction}

\begin{prop} The map $g: \Chat\rightarrow\Chat$ is quasiregular. 
\end{prop}

\begin{proof} As in the proof of Proposition \ref{gisqr1} this follows since $g$ is quasiregular in each relevant region, and continuous across any common boundary.
\end{proof}

\begin{definition} We define a Beltrami coefficient $\mu$ in $\Chat$ as follows:
\begin{enumerate}
\item $\mu|_{A_1\cup A_0}:=0$,
\item $\mu|_{\sigma_f^{-1}(A)}:=h^*(\mu|_{A_1\cup A_0})=h^*(0)$,
\item Pull back $\mu|_{\sigma_f^{-1}(A)}$ under $\sigma_f^n$ to define $\mu$ in subsequent $\sigma_f$-pullbacks of $\sigma_f^{-1}(A)$
\item $\mu:=0$ elsewhere in $\mathbb{D}$, and
\item $\mu(z):=\frac{z^2}{\overline{z}^2}\overline{\mu(1/\overline{z})}$ for $z\in\mathbb{D}^*$.
\end{enumerate}
\end{definition}

\begin{rem}\label{symmetryform} The formula $\mu(z)=\frac{z^2}{\overline{z}^2}\overline{\mu(1/\overline{z})}$ for $z\in\mathbb{D}^*$ will ensure that the solution to the Beltrami equation with coefficient $\mu$ preserves $\mathbb{T}$ (see Section 4.2 of \cite{MR3432154}).
\end{rem}

\begin{prop}\label{g_invariancef} The Beltrami coefficient $\mu$ is $g$-invariant, and $||\mu||_{L^\infty}<1$. 
\end{prop}

\begin{proof} The proof is the same as that of Proposition \ref{g_invariance}.
\end{proof}

\begin{notation} By existence of solutions to the Beltrami equation, there exists a quasiconformal mapping we denote $\phi: \Chat\rightarrow\Chat$ so that $\phi_{\overline{z}}=\mu\phi_z$, and we choose $\phi$ so that $\phi(\mathbb{T})=\mathbb{T}$ and $\phi(\mathbb{D})=\mathbb{D}$ (see Remark \ref{symmetryform}). We set 
\[   B:=  \phi \circ g \circ \phi^{-1}. \] 
\end{notation}

\begin{thm}\label{p_propf} The map $B$ is a Blaschke product with an attracting fixed point in $\mathbb{D}$.
\end{thm}

\begin{proof}  The map $B$ is a Blaschke product since it is rational and fixes the unit circle. Moreover, it is readily verified that $B$ has an attracting fixed point at $\phi(z_0)$.
\end{proof}

\begin{thm}\label{tilingset=disc} We have $T(\sigma_f)=\mathbb{D}$.
\end{thm}

\begin{proof} Recall that since Beltrami coefficient $\mu$ is $\mathbb{T}$-symmetric, so is the map $\phi$. In other words, $\phi(\mathbb{T})=\mathbb{T}$ and $\phi(\mathbb{D})=\mathbb{D}$. Recall that $\mathbb{T}=\mathcal{J}(B)$ (since $B$ has an attracting fixed point in $\mathbb{D}$), and for any $z\in\mathbb{D}$ the orbit of $z$ converges to the attracting fixed point $z_0$. Since $\phi$ conjugates $B$ to $g$ and $\phi(\mathbb{D})=\mathbb{D}$, it follows that the orbit of any $z\in\mathbb{D}$ under $g$ also converges to $z_0\in A_0$. As $g=\sigma_f$ in $\mathbb{D}\setminus (A_1\cup A_0 \cup \sigma_f^{-1}(A))$, it follows that the orbit of any $z\in\mathbb{D}$ under $\sigma_f$ eventually lands in $T^0(\sigma_f)$. In other words, $\mathbb{D}\subset T(\sigma_f)$, and since $T(\sigma_f)\subset\mathbb{D}$, the result follows.
\end{proof}

\section{Welding $\sigma_f$, $p$.}\label{bftor}

In this Section we prove Theorem \ref{mating_thm2}. We start by welding a given $\sigma_f$, $p$ in the following Construction. For a Jordan curve $\gamma\subset\mathbb{C}$, we let the bounded component of $\Chat\setminus\gamma$ be denoted by $\interior(\gamma)$. 

\begin{construction}\label{section5const} Let $\sigma_f\in\Sigma^\mathcal{F}_d$ and $p\in\overline{\poldmh}$. Assume without loss of generality that $0\in T^0(\sigma_f)$. Since $\sigma_f$ is expanding on $\mathbb{T}$, there exists $r<1$ so that:
\begin{enumerate} 
\item $\sigma_f(r\mathbb{T})$ is properly contained in $r\mathbb{D}$, and
\item $r\mathbb{D}$ contains $T^0(\sigma_f)\cup\CV(\sigma_f)$.
\end{enumerate}
Consider the annulus $A$ with inner boundary $r\mathbb{T}$ and outer boundary $\sqrt[d-1]{r}\mathbb{T}$. As in Construction \ref{sigmaqrconstr}, there exists a quasiregular interpolation $h$ in $A$ between $\sigma_f|_{r\mathbb{T}}$ and $\overline{z}^{d-1}|_{\sqrt[d-1]{r}\mathbb{T}}$. We define the map
\begin{equation}
G(z):=\begin{cases} \sigma_f(z) \textrm{ in } r\mathbb{D}\setminus T^0(\sigma_f),  \\  h(z) \textrm{ in } A,  \\  \overline{z}^{d-1} \textrm{ in } \mathbb{D}\setminus\sqrt[d-1]{r}\mathbb{D}  \\  \end{cases}
\end{equation}
We let $\psi: \mathcal{A}(\infty)\rightarrow\mathbb{D}$ denote a $\mathbb{D}$-valued B\"ottcher coordinate for the basin of attraction for $\infty$ of $p$, with $\psi(\infty)=0$, and we define the map
\begin{equation}
g(z):=\begin{cases} p \textrm{ in } K(p),  \\  \psi^{-1}\circ G\circ\psi \textrm{ in } \mathcal{A}(\infty).  \end{cases}
\end{equation}
\end{construction}

\begin{prop}\label{gisqr} The map $g: \Chat\setminus \psi^{-1}(T^0(\sigma_f)) \rightarrow\Chat$ is quasiregular. 
\end{prop}

\begin{proof} The map $G$ is quasiregular in each of the regions in the piecewise-definition and continuous across the common boundaries, hence by removability of analytic arcs for quasiregular mappings, $G$ is quasiregular on $\mathbb{D}$ and $g$ is quasiregular on $\Chat$. 
\end{proof}

The next Definition \ref{beltcoefflasttime} and Proposition \ref{g_invariance_sigmaf} mimic Definition \ref{gisqrsectbelt} and Proposition \ref{g_invariance} almost exactly, and so we omit the proof of Proposition \ref{g_invariance_sigmaf}.

\begin{definition}\label{beltcoefflasttime} We define a Beltrami coefficient $\nu$ in $\mathbb{D}$ as follows:
\begin{enumerate}
\item $\nu:=0$ in $r\mathbb{D}$,
\item $\nu|_{A}:=h^*(\nu|_{r\mathbb{D}})=h^*(0)$, 
\item pull back $\nu$ under $\overline{z}^{d-1}$ to define $\nu$ throughout $\mathbb{D}\setminus(r\mathbb{D}\cup A)$, and
\item $\nu:=0$ elsewhere in $\mathbb{D}$.
\end{enumerate}
We define the Beltrami coefficient $\mu$ in $\Chat$ by $\mu:=\psi^*(\nu)$ in $\mathcal{A}(\infty)$ (the pullback of $\nu$ under $\psi: \mathcal{A}(\infty)\rightarrow\mathbb{D}$) and $\mu:=0$ otherwise.
\end{definition}

\begin{prop}\label{g_invariance_sigmaf} The Beltrami coefficient $\mu$ is $g$-invariant, and $||\mu||_{L^\infty}<1$. 
\end{prop}

\begin{notation} By existence of solutions to the Beltrami equation, there exists a quasiconformal mapping we denote $\phi: \Chat\rightarrow\Chat$ so that $\phi_{\overline{z}}=\mu\phi_z$, and we normalize $\phi$ so that $\phi(0)=\infty$ and $\phi(\infty)=0$. We set 
\[   \sigma(z):=  \phi \circ g \circ \phi^{-1}(z) \textrm{ for } z\in\Chat\setminus\phi\circ\psi^{-1}(T^0(\sigma_f)). \] 
\end{notation}

\begin{thm} We have $\sigma\in\Sigma_d$. In other words, there exists $r\in\ratd$ univalent in a neighborhood of a Euclidean disc $D$ with $\sigma_{r,D}=\sigma$ satisfying $K(\sigma)$ is connected, full and the critical values of $\sigma$ in $K(\sigma)$ are attracted to attracting cycles.
\end{thm}

\begin{proof} First note that by $g$-invariance of $\mu$, the map $\sigma$ is anti-holomorphic on its domain of definition. Moreover, $G$ is the identity on $\partial T^0(\sigma_f)$, and so $g$ is the identity on $\psi^{-1}(\partial T^0(\sigma_f))$, and so in turn we have that $\sigma$ is the identity on $\phi\circ\psi^{-1}(\partial T^0(\sigma_f))$.

Let $r: \mathbb{D}^* \rightarrow  \Chat\setminus \phi\circ\psi^{-1}( T^0(\sigma_f))$ be a Riemann map normalized so that $r(\infty)=\infty$. We apply the Schwarz reflection principle to extend $r$ to a self-map of $\Chat$: 
\[  r(z):=     \sigma\circ r\left(\frac{1}{\overline{z}}\right)  \textrm{ for } z\in\mathbb{D}.  \]
Since $\sigma$ is anti-holomorphic, the map $r$ is a holomorphic self-map of $\Chat$ hence must be rational, and it is readily checked that $\deg(r)=d$. Since $\sigma$, $\sigma_r$ are both the identity on $\phi\circ\psi^{-1}(\partial T^0(\sigma_f))$, we conclude that $\sigma\equiv\sigma_r$. Since $p\in\overline{\poldmh}$ and $\phi$ conjugates $\sigma|_{K(\sigma)}$ to $p|_{K(p)}$ (and this conjugacy is conformal on the interior) it follows that the critical values of $\sigma$ in $K(\sigma)$ are attracted to attracting cycles of $\sigma$ (since $p$ has this property). Moreover, $r$ is conformal in a neighborhood of $\mathbb{D}^*$ since $f$ is conformal in a neighborhood of $\overline{A_{\sqrt{s}}}$. Lastly, since $\sigma=\sigma_{r}$ we have $K(\sigma)=K(\sigma_r)=\phi^{-1}(K(p))$ is connected and full since $K(p)$ is connected and full, and so $\sigma\in\Sigma_d$.
\end{proof}

\begin{thm}\label{conversemainthm} The map $\sigma$ is a conformal mating of $p$ with $\sigma_f$.
\end{thm}

\begin{proof} It is readily verified that $\psi_p:=\phi^{-1}$ is the desired conjugacy between $\sigma|_{K(\sigma)}$ and $p|_{K(p)}$. 

Let us now show that $\sigma$ is conjugate to $\sigma_f$ on $T(\sigma)$. First note that since $\mu=0$ in $r\mathbb{D}$, we have $\Phi:=\psi\circ\phi^{-1}: \phi\circ\psi^{-1}(r\mathbb{D}) \rightarrow r\mathbb{D}$ is a conformal conjugacy between $\sigma$ and $\sigma_f$. In $T(\sigma)\setminus\phi\circ\psi^{-1}(r\mathbb{D})$ (resp. $\{z : r<|z|<1\}$) the map $\sigma$ (resp. $\sigma_f$) is a covering map, and so the conformal conjugacy $\Phi$ may be pulled back under iterates of $\sigma_f$, $\sigma$ to yield a full conjugacy between $\sigma|_{T(\sigma)}$ and $\sigma_f|_{\mathbb{D}}$ (recall Theorem \ref{tilingset=disc}). 

After normalizing the $\mathbb{D}$-valued B\"ottcher coordinate in Construction \ref{section5const} appropriately, Condition (C) in Definition \ref{conf_mat_defn} is also readily verified since then $\beta^{-1}\circ\psi_p\circ\psi_f^{-1}$ is a conjugacy between $\sigma_f|_{\mathbb{T}}$ and $\overline{z}^{d-1}$ fixing $1$. 
\end{proof}

\noindent Let us now move on to the question of uniqueness of $\sigma$.

\begin{thm}\label{injectivity_statement} Suppose two Schwarz reflections $\sigma$, $\tilde{\sigma}$ are both conformal matings of the same $p$, $\sigma_f$. Then $\sigma=M\circ \tilde{\sigma}\circ M^{-1}$ for some M\"obius $M$.
\end{thm}

\begin{proof} Let $\psi_p$, $\psi_f$ conjugate $\sigma$ to $p$, $\sigma_f$ and let $\tilde{\psi}_p$, $\tilde{\psi}_f$ conjugate $\tilde{\sigma}$ to $p$, $\sigma_f$. Consider the piecewise-defined
\begin{equation}\label{Mdefnn} M:= \begin{cases} \tilde{\psi}_p^{-1} \circ \psi_p  \textrm{ in }  K(\sigma)    \\     \tilde{\psi}_f^{-1} \circ \psi_f    \textrm{ in }  T(\sigma).     \end{cases} 
\end{equation}
The common boundary of $K(\sigma)$, $T(\sigma)$ is $\mathcal{J}(\sigma)$ (see Corollary \ref{Jordan_curve_prop}); let us show the two piecewise-definitions agree across $\mathcal{J}(\sigma)$, or, equivalently, we need to show that
\begin{equation}\label{wts==} \tilde{\psi}_p\circ \tilde{\psi}_f^{-1} = \psi_p\circ \psi_f^{-1} \textrm{ on } \mathbb{T}. \end{equation}
Well (\ref{wts==}) holds according to the Definition \ref{conf_mat_defn} of conformal mating, and thus we conclude that $M$ is continuous across $\mathcal{J}(\sigma)$. Note that $\mathcal{J}(\sigma)$ is removable since $\mathcal{J}(\sigma)$ is quasiconformally homeomorphic to $\mathcal{J}(p)$, and $\mathcal{J}(p)$ is removable since $p$ is hyperbolic. Thus $M$ is conformal across $\mathcal{J}(\sigma)$ and hence a M\"obius transformation. From (\ref{Mdefnn}) we see that $M$ conjugates $\sigma$, $\tilde{\sigma}$ as needed.
 \end{proof}

\noindent Together, Theorems \ref{conversemainthm} and \ref{injectivity_statement} imply Theorem \ref{mating_thm2}.


\section{Degeneration}\label{degen_sec}

\begin{notation}\label{degeneration_section_setup} Let $r\in\ratd$ be univalent in a Euclidean disc $D$ and assume $r$ has $d$ critical points on $\partial D$. Following \cite{MR4706575}, the \emph{desingularized} droplet is defined to be 
\[ T^0(r):= \Chat\setminus \left[r(\interior(D))\cup r(\CP(r)\cap\partial D)\right], \]
and the \emph{Tiling set} of $\sigma_r$ is defined to be $T(\sigma_r):=\cup_{n=0}^\infty\sigma^{-n}(T^0(r))$ (we remark that $T(\sigma_r)$ is open: see for instance Proposition 4.6 of \cite{MR4706575}). The \emph{Filled Julia set} of $\sigma_r$ is defined to be $K(\sigma):=\Chat\setminus T(\sigma_r)$. We will assume that $\sigma_r: K(\sigma) \rightarrow K(\sigma)$ is conjugate to a hyperbolic polynomial, and the conjugacy is conformal on $\interior K(\sigma)$.  
\end{notation}

\begin{notation}\label{degeneration_section_setup2} Let $(r_n)_{n=1}^\infty\in\ratd$ be a sequence so that $r_n\rightarrow r$ uniformly on $\Chat$ with respect to the spherical metric, and each $r_n$ is univalent in a neighborhood of $D$ and moreover $\sigma_{r_n}\in\Sigma_d$ (whereas $\sigma_r\not\in\Sigma_d$). We will abbreviate $\sigma:=\sigma_r$ and $\sigma_n:=\sigma_{r_n}$
\end{notation}

\begin{example} The maps 
\[ r(z):=z+\frac{1}{(d-1)z^{d-1}}\textrm{, and } r_n(z):=z+\frac{1}{(d-1+\frac{1}{n})z^{d-1}} \]
satisfy the conditions in Notation \ref{degeneration_section_setup2}. 
\end{example}

\noindent We will prove that the tiling sets $T(\sigma_{n})$ converge to $\interior T(\sigma)$ in the Carath\'eodory sense; let us first recall the notion of Carath\'eodory convergence.


\begin{definition}\label{kerneldefn} Let $z_0\in\Chat$, and $\Omega_n$ be a sequence of domains all containing $z_0$. Consider 
\[ E:=\{z :  \textrm{ there exists a neighborhood } U \ni z \textrm{ so that } U \subset \Omega_n \textrm{ for all large } n\}.\] 
The \emph{kernel} of $\Omega_n$ with respect to $z_0$ is defined to be the connected component of $E$ containing $z_0$. We say that a sequence of domains $\Omega_n$ converges to a domain $\Omega$ in the \emph{Carath\'eodory sense} if $\Omega$ is the kernel of every subsequence of $\Omega_n$.
\end{definition}

\begin{rem} The statement of Theorem \ref{reconciliation_thm} assumes that $\interior T^0(\sigma)$ is connected; this need not be the case in general. Indeed, if $\sigma=\sigma_r$ with $r(\partial D)$ containing double points (meaning $r(\partial D)$ is not a Jordan curve), then $\interior T^0(\sigma)$ is disconnected. The presence of double points presents no serious conceptual difficulty in the proof of Theorem \ref{reconciliation_thm}, but would require a separate analysis and statement.  
\end{rem}

\begin{rem} We will henceforth fix a point $z_0$ so that $z_0\in \interior T^0(\sigma)\cap T^0(\sigma_n)$ for all large $n$. 
\end{rem}

\begin{thm}\label{tilingsetconv} The tiling sets $T(\sigma_{n})$ converge to $T(\sigma)$ in the Carath\'eodory sense. 
\end{thm}

\begin{proof} Let $\Omega_n:=T(\sigma_{n})$. Consider $E$ as in Definition \ref{kerneldefn}. We will show that $E=T(\sigma)$.

First let us first show that $T(\sigma)\subset E$. Let $z\in T(\sigma)$. Thus $\sigma^k(z)\in T^0(\sigma)$ for some $k$. Thus either  $\sigma^k(z)\in\interior T^0(\sigma)$ or $\sigma^k(z)\in T^0(\sigma)\setminus \interior T^0(\sigma)$. Suppose first that $\sigma^k(z)\in\interior T^0(\sigma)$. Since $r_n\rightarrow r$, we have that $\sigma^k(z)\in\interior T^0(\sigma_n)$ for sufficiently large $n$. Thus since $r_n\rightarrow r$, we have $\sigma_n^k(z)\in\interior T^0(\sigma_n)$ for sufficiently large $n$, and so $z\in E$. Now suppose $\sigma^k(z)\in T^0(\sigma)\setminus \interior T^0(\sigma)$. Then $\sigma^k(z)\in B_1(\sigma):=\interior(T^0(\sigma)\cup T^1(\sigma))$. Again since $r_n\rightarrow r$, we have that $\sigma^k(z)\in B_1(\sigma_n):=\interior(T^0(\sigma_n)\cup T^1(\sigma_n))$ for sufficiently large $n$. Thus since $r_n\rightarrow r$, we have $\sigma_n^k(z)\in B_1(\sigma_n)$ for sufficiently large $n$ and so $z\in E$.

Let us now prove $E\subset T(\sigma)$. We will show the contrapositive, namely we will assume that $z\not\in T(\sigma)$ and show that this implies $z\not \in E$. Since $z\not\in T(\sigma)$, either $z\in\partial T(\sigma)$ or $z\in\interior(K(\sigma))$. Let us first consider the case $z\in\interior K(\sigma)$. Recall we have assumed in Notation \ref{degeneration_section_setup} that $\sigma: \interior K(\sigma) \rightarrow \interior K(\sigma)$ is conjugate to a hyperbolic polynomial on its filled Julia set. Thus there exists an attracting cycle of $\sigma$ to which $z$ is attracted. Since $r_n\rightarrow r$, we also have that $\sigma_n(z)$ is attracted to an attracting cycle of $\sigma_n$ for large enough $n$. In other words $z\not\in T(\sigma_{n})$ for large $n$ and so $z\not \in E$.  Now assume $z\not\in T(\sigma)$ but $z\in\partial T(\sigma)$. Then any neighborhood $U$ of $z$ contains points $w\in\interior K(\sigma)$, so applying the same arguments as above to $w$ shows that $U$ contains points in $\interior K(\sigma_n)$. Thus $U\not\subset T(\sigma_{n})$ for all large $n$. Since $U$ was arbitrary, this shows $z\not\in E$.
\end{proof}

\begin{notation} Recalling $T(\sigma)$, $T(\sigma_n)$ are all open and simply connected, we denote by $\psi: \mathbb{D} \rightarrow T(\sigma)$ and $\psi_n: \mathbb{D}\rightarrow T(\sigma_{n})$ the unique conformal mappings which map $0$ to $z_0$ and with positive derivative at $0$.
\end{notation}

\begin{thm}\label{conj_conv} The maps $\psi_n$ converge uniformly on compact subsets of $\mathbb{D}$ to $\psi$, and the inverse maps $\psi_n^{-1}$ converge uniformly on compact subsets of $T(\sigma)$ to $\psi^{-1}$.
\end{thm}

\begin{proof} This follows from Theorem \ref{tilingsetconv} and the Carath\'eodory Kernel Convergence theorem. 
\end{proof}

\begin{notation} We set 
\[ D_n:= \mathbb{D}\setminus\psi_n^{-1}(T^0(\sigma_n)) \textrm{ and } \sigma_{f_n}:=  \psi_n^{-1} \circ \sigma_n  \circ \psi_n : D_n\rightarrow \mathbb{D}, \]
and 
\[ D:= \mathbb{D}\setminus\psi^{-1}(T^0(\sigma)), \textrm{ and } \rho:=  \psi^{-1} \circ \sigma  \circ \psi : D\rightarrow \mathbb{D} \]
\end{notation}

\begin{rem} By pre-composing $\psi$ by a rotation, we may assume we have $\rho(1)=1$, and so pre-composing the $\psi_n$ by rotations we may ensure that $\sigma_{f_n}(1)=1$ for all $n$ and still the convergence of Theorem \ref{conj_conv} holds.
\end{rem}

\begin{prop} We have $\sigma_{f_n}\in\Sigma_d^\mathcal{F}$, and $\rho$ is a reflection map for some circle reflection group as in Definition \ref{necklace}. Moreover, $\sigma_n$ is the conformal mating of $\sigma_{f_n}$ with a hyperbolic polynomial, and $\sigma$ is the conformal mating of $\rho$ with a hyperbolic polynomial. 
\end{prop}

\begin{proof} We have $\sigma_{f_n}\in\Sigma_d^\mathcal{F}$ by Theorem \ref{exist_of_f}. That $\rho$ is a Nielsen for some circle reflection group as in Definition \ref{necklace} is proved for instance in \cite{MR4706575}, \cite{MR4961627}; one simply finds a circle reflection group with fundamental region conformally equivalent to $T^0(\sigma)$ (with vertices $\mapsto$ vertices) and lifts this conformal map under $\sigma$ and reflections in the circles. The last statement of the Proposition follows from the definition of conformal mating. 
\end{proof}

\begin{prop} The maps $\sigma_{f_n}$ converge to $\rho$ uniformly on compact subsets of $D$.
\end{prop}

\begin{proof} This follows since $r_n\rightarrow r$ and by Theorem \ref{conj_conv}. 
\end{proof}

\noindent \emph{Proof of Theorem \ref{reconciliation_thm}.} Recall $\sigma_{f_n}$ is conjugate via $\psi_n$ to $r_n$ and $\sigma$ is conjugate via $\psi$ to $r$. Since $r_n\rightarrow r$ and $\psi_n\rightarrow \psi$, $\psi_n^{-1}\rightarrow \psi^{-1}$ by Theorem \ref{conj_conv}, the conclusion follows. \qed


\begin{thebibliography}{LLMM23b}

\bibitem[AIPP15]{MR3432154}
Kari Astala, Oleg Ivrii, Antti Per\"{a}l\"{a}, and Istv\'{a}n Prause.
\newblock Asymptotic variance of the {B}eurling transform.
\newblock {\em Geom. Funct. Anal.}, 25(6):1647--1687, 2015.

\bibitem[BF14]{MR3445628}
Bodil Branner and N\'{u}ria Fagella.
\newblock {\em Quasiconformal surgery in holomorphic dynamics}, volume 141 of
  {\em Cambridge Studies in Advanced Mathematics}.
\newblock Cambridge University Press, Cambridge, 2014.
\newblock With contributions by Xavier Buff, Shaun Bullett, Adam L. Epstein,
  Peter Ha\"{\i}ssinsky, Christian Henriksen, Carsten L. Petersen, Kevin M.
  Pilgrim, Tan Lei and Michael Yampolsky.

\bibitem[BLLM26]{MR5049498}
Shaun Bullett, Luna Lomonaco, Mikhail Lyubich, and Sabyasachi Mukherjee.
\newblock Mating parabolic rational maps with {H}ecke groups.
\newblock {\em Proc. Lond. Math. Soc. (3)}, 132(3):Paper No. e70132, 43, 2026.

\bibitem[DH85]{MR816367}
Adrien Douady and John~Hamal Hubbard.
\newblock On the dynamics of polynomial-like mappings.
\newblock {\em Ann. Sci. \'Ecole Norm. Sup. (4)}, 18(2):287--343, 1985.

\bibitem[Gus90]{MR1094715}
Bj\"orn Gustafsson.
\newblock On quadrature domains and an inverse problem in potential theory.
\newblock {\em J. Analyse Math.}, 55:172--216, 1990.

\bibitem[LLM22a]{MR4365080}
Russell Lodge, Yusheng Luo, and Sabyasachi Mukherjee.
\newblock Circle packings, kissing reflection groups and critically fixed
  anti-rational maps.
\newblock {\em Forum Math. Sigma}, 10:Paper No. e3, 38, 2022.

\bibitem[LLM22b]{MR4536467}
Russell Lodge, Yusheng Luo, and Sabyasachi Mukherjee.
\newblock On deformation space analogies between {K}leinian reflection groups
  and antiholomorphic rational maps.
\newblock {\em Geom. Funct. Anal.}, 32(6):1428--1485, 2022.

\bibitem[LLM24]{luo2024generaldynamicaltheoryschwarz}
Yusheng Luo, Mikhail Lyubich, and Sabyasachi Mukherjee.
\newblock A general dynamical theory of schwarz reflections, b-involutions, and
  algebraic correspondences, 2024.

\bibitem[LLMM23a]{MR4706575}
Seung-Yeop Lee, Mikhail Lyubich, Nikolai~G. Makarov, and Sabyasachi Mukherjee.
\newblock Dynamics of {S}chwarz reflections: the mating phenomena.
\newblock {\em Ann. Sci. \'Ec. Norm. Sup\'er. (4)}, 56(6):1825--1881, 2023.

\bibitem[LLMM23b]{MR4543152}
Russell Lodge, Mikhail Lyubich, Sergei Merenkov, and Sabyasachi Mukherjee.
\newblock On dynamical gaskets generated by rational maps, {K}leinian groups,
  and {S}chwarz reflections.
\newblock {\em Conform. Geom. Dyn.}, 27:1--54, 2023.

\bibitem[LLMM25]{MR4940720}
Seung-Yeop Lee, Mikhail Lyubich, Nikolai~G. Makarov, and Sabyasachi Mukherjee.
\newblock Schwarz reflections and the tricorn.
\newblock {\em Ann. Inst. Fourier (Grenoble)}, 75(5):1987--2100, 2025.

\bibitem[LM16]{MR3454377}
Seung-Yeop Lee and Nikolai~G. Makarov.
\newblock Topology of quadrature domains.
\newblock {\em J. Amer. Math. Soc.}, 29(2):333--369, 2016.

\bibitem[LM26a]{sabyaicmpaper}
Luna Lomonaco and Sabyasachi Mukherjee.
\newblock Algebraic correspondences and schwarz reflections: Where rational
  dynamics meets kleinian groups.
\newblock In {\em Proceedings of the International Congress of Mathematicians},
  volume~5, page 168–190, 2026.

\bibitem[LM26b]{MR5098194}
Mikhail Lyubich and Sabyasachi Mukherjee.
\newblock Mirrors of conformal dynamics: interplay between anti-rational maps,
  reflection groups, {S}chwarz reflections, and correspondences.
\newblock In {\em Algebraic, complex, and arithmetic dynamics}, Simons Symp.,
  pages 341--476. Springer, Cham, [2026] \copyright 2026.

\bibitem[LMM21]{MR4273178}
Kirill Lazebnik, Nikolai~G. Makarov, and Sabyasachi Mukherjee.
\newblock Univalent polynomials and {H}ubbard trees.
\newblock {\em Trans. Amer. Math. Soc.}, 374(7):4839--4893, 2021.

\bibitem[LMM22]{MR4381220}
Kirill Lazebnik, Nikolai~G. Makarov, and Sabyasachi Mukherjee.
\newblock Bers slices in families of univalent maps.
\newblock {\em Math. Z.}, 300(3):2771--2808, 2022.

\bibitem[LMM24]{MR4806290}
Mikhail Lyubich, Jacob Mazor, and Sabyasachi Mukherjee.
\newblock Antiholomorphic correspondences and mating {I}: {R}ealization
  theorems.
\newblock {\em Commun. Am. Math. Soc.}, 4:495--547, 2024.

\bibitem[LMMN25]{MR4961627}
Mikhail~Yu. Lyubich, Sergei Merenkov, Sabyasachi Mukherjee, and Dimitrios
  Ntalampekos.
\newblock David extension of circle homeomorphisms, welding, mating, and
  removability.
\newblock {\em Mem. Amer. Math. Soc.}, 313(1588):v+110, 2025.

\bibitem[LN26]{LUO_NTALAMPEKOS_2026}
Yusheng Luo and Dimitrios Ntalampekos.
\newblock Piecewise quasiconformal dynamical systems of the unit circle.
\newblock {\em Ergodic Theory and Dynamical Systems}, page 1–41, 2026.

\bibitem[Mil06]{MilnorCDBook}
John Milnor.
\newblock {\em Dynamics in one complex variable}, volume 160 of {\em Annals of
  Mathematics Studies}.
\newblock Princeton University Press, Princeton, NJ, third edition, 2006.

\bibitem[MM22]{MR4472055}
Mahan Mj and Sabyasachi Mukherjee.
\newblock Combination theorems in groups, geometry and dynamics.
\newblock In {\em In the tradition of {T}hurston {II}. {G}eometry and groups},
  pages 331--383. Springer, Cham, [2022] \copyright 2022.

\bibitem[MM23a]{MR4588722}
Mahan Mj and Sabyasachi Mukherjee.
\newblock Combining rational maps and {K}leinian groups via orbit equivalence.
\newblock {\em Proc. Lond. Math. Soc. (3)}, 126(5):1740--1809, 2023.

\bibitem[MM23b]{MR4667419}
Mahan Mj and Sabyasachi Mukherjee.
\newblock The {S}ullivan dictionary and {B}owen-{S}eries maps.
\newblock {\em EMS Surv. Math. Sci.}, 10(1):179--221, 2023.

\bibitem[MM25]{MR4961767}
Mahan Mj and Sabyasachi Mukherjee.
\newblock Matings, holomorphic correspondences, and a {B}ers slice.
\newblock {\em J. \'{E}c. polytech. Math.}, 12:1445--1502, 2025.

\bibitem[PM12]{MR3088260}
Carsten~Lunde Petersen and Daniel Meyer.
\newblock On the notions of mating.
\newblock {\em Ann. Fac. Sci. Toulouse Math. (6)}, 21(5):839--876, 2012.

\bibitem[Put94]{MR1302653}
Mihai Putinar.
\newblock On a class of finitely determined planar domains.
\newblock {\em Math. Res. Lett.}, 1(3):389--398, 1994.

\bibitem[Ric72]{Ric72}
S.~Richardson.
\newblock {Hele Shaw} flows with a free boundary produced by the injection of
  fluid into a narrow channel.
\newblock {\em Journal of Fluid Mechanics}, 56(4):609--618, 1972.

\bibitem[RM25]{rashmita2025topologysingularitiesquadraturedomains}
Rashmita and Sabyasachi Mukherjee.
\newblock On topology and singularities of quadrature domains, 2025.

\bibitem[WY17]{MR3716945}
Xiaoguang Wang and Yongcheng Yin.
\newblock Global topology of hyperbolic components: {C}antor circle case.
\newblock {\em Proc. Lond. Math. Soc. (3)}, 115(4):897--923, 2017.

\bibitem[WZ03]{MR1986427}
P.~Wiegmann and A.~Zabrodin.
\newblock Large scale correlations in normal non-{H}ermitian matrix ensembles.
\newblock volume~36, pages 3411--3424. 2003.
\newblock Random matrix theory.

\end{thebibliography}
\end{document}